\documentclass[12pt]{amsart}
\usepackage{graphicx}
\graphicspath{{figures/}}
\usepackage{enumitem}

\usepackage{comment}

\usepackage{todonotes}

\usepackage{amssymb}
\usepackage{graphicx}
\usepackage{epsfig}
\usepackage{subcaption}
\theoremstyle{plain}
\newtheorem{theorem}{Theorem}[section]
\newtheorem{lemma}[theorem]{Lemma}
\newtheorem{proposition}[theorem]{Proposition}
\newtheorem{corollary}[theorem]{Corollary}
\theoremstyle{definition}

\newtheorem{remark}[theorem]{Remark}

\numberwithin{equation}{section}

\newcommand{\R}{\mathbb{R}}
\newcommand{\Z}{\mathbb{Z}}

\newcommand{\C}{\mathbb{C}}
\newcommand{\Q}{\mathbb{Q}}
\newcommand{\N}{\mathbb{N}}

\newcommand{\ol}{\overline}

\renewcommand{\Im}{\mathrm{Im}}
\renewcommand{\Re}{\mathrm{Re}}

\title[Renormalization of Translated Cone Exchange Transformations]{Renormalization of Translated Cone Exchange Transformations}

\author[Noah Cockram, Peter Ashwin and Ana Rodrigues]{Noah Cockram$^1$, Peter Ashwin$^1$ and Ana Rodrigues$^{1,2,3}$}

\address{$^1$ Department of Mathematics and Statistics, University of Exeter, Exeter EX4 4QE, UK} 

\address{$^2$ Departamento de Matem\'{a}tica, Escola de Ci\^{e}ncias  e Tecnologia, Universidade de \'{E}vora, Rua Rom\~{a}o Ramalho, 59, 7000--671 \'{E}vora, Portugal }

\address{$^3$ Centro de Investiga\c{c}\~{a}o em Matem\'{a}tica e Aplica\c{c}\~{o}es, Rua Rom\~{a}o  Ramalho, 59, 7000--671 \'{E}vora, Portugal} 

\begin{document}

\begin{abstract}
In this paper, we investigate a class of non-invertible piecewise isometries on the upper half-plane known as Translated Cone Exchanges.  These maps include a simple interval exchange on a boundary we call the baseline. We provide a geometric construction for the first return map to a neighbourhood of the vertex of the middle cone for a large class of parameters, then we show a recurrence in the first return map tied to Diophantine properties of the parameters, and subsequently prove the infinite renormalizability of the first return map for these parameters.
\end{abstract}

\maketitle

\section{Introduction}

\textit{Piecewise isometries (PWIs)} are a class of maps that can be generally described as a ``cutting-and-shuffling'' action of a metric space, specifically a partitioning of the phase space into at most countably many convex pieces called \textit{atoms}, which are each moved according to an isometry.  The phase space of these maps can be partitioned into two (or three) subsets based on the dynamics -- a polygon or disc packing of periodic islands known as the \textit{regular set}, and its complement, the set of points whose orbit either lands on, or accumulates on, the discontinuity set.  Some authors choose to further distinguish those points in the pre-images of the discontinuity and those points which accumulate on it.  
The most well-known and well-understood examples of such maps are the \textit{interval exchange transformations (IETs)}, which arise as return maps to cross-sections of some measured foliations \cite{M1} and also as generalisations of circle rotations \cite{BC, Z1, LM} and their encoding spaces generalise Sturmian shifts \cite{FZ1}.  Furthermore, interval exchanges which aren't irrational rotations are known to be almost always weakly mixing \cite{AG1} but never strongly mixing \cite{KA}.  Piecewise isometries in general, however, are not as well-known and as a subset of this class, interval exchanges are in many ways exceptional, due in part to being one-dimensional, as well as the invariance of Lebesgue measure.

In the more general setting, although the inherent lack of hyperbolicity restricts the variety of possible behaviours, for example it is known that all piecewise isometries have zero topological entropy \cite{B01}, piecewise isometries are still capable of quite complex behaviour; many examples show the presence of unbounded periodicity and an underlying renormalizability which structures the dynamics near the discontinuities \cite{GA, GP, AKT, LKV, AG05, AG06, P06}; numerical evidence suggests the existence of invariant curves in the exceptional set which seem fractal-like and form barriers to ergodicity \cite{AG05, AG06, AGPR, LOUL}; there are conjectured conditions for piecewise isometries to have sensitive dependence on initial conditions \cite{BK09}.

Renormalization in theoretical physics and nonlinear dynamical systems has a longstanding history, see for example \cite{W83, CT78, F78, F79, M94, R79, V82}, driven by the problem of understanding phenomena that occur simultaneously at many spatial and temporal scales, particularly near phase transitions, periodic points, or in the case of piecewise isometries, the set of discontinuities.

In this paper, we investigate the renormalizability of a class of piecewise isometries called Translated Cone Exchanges on the closure of the upper half-plane $\ol{\mathbb{H}}$.  This family of maps was introduced in \cite{AGPR} and has since been investigated in \cite{P19,PR18}. In particular, we use a geometric construction to describe the action of a first return map to a subset containing the origin, and show that this map displays renormalizable behaviour locally to the origin in accordance with Diophantine approximation of one of its parameters. These results go beyond \cite{P19,PR18} in that they are much less constrained in the continued fraction expansion associated with the baseline translation.

This paper is organized as follows. In Section 2, we introduce the family of maps we will investigate, namely, Translated Cone Exchange transformations. In Section 3 we will develop some tools that will be useful in the next section. Section 4 presents the preliminary results that lead to the main result of this paper, Theorem \ref{Renorm}, which gives an explicit form of renormalization for the first return maps of maps in our class to a neighbourhood of 0. Finally, in Section 5 we present an example for fixed values of the parameters.

\section{Translated Cone Exchange transformations}

Let $\mathbb{H} \subset \C$ denote the upper half plane, and let $\ol{\mathbb{H}}$ be its closure in $\C$, that is
\begin{equation*}
\ol{\mathbb{H}} = \{z \in \C: \Im(z) \geq 0\}.
\end{equation*}
A \textit{Translated Cone Exchange transformation (TCE)} \cite{AGPR} is a PWI $(\mathcal{C}, F_\kappa)$ defined on the closed upper half plane $\ol{\mathbb{H}}$.  For any integer $d > 0$, let $\mathbb{B}^{d+2}$ be the set
\begin{equation*}
\mathbb{B}^{d+2} = \left\{ \alpha = (\alpha_0,...,\alpha_{d+1}) \in (0, \pi)^{d+2} : \sum_{j=0}^{d+1} \alpha_j = \pi \right\},
\end{equation*}
Next, for some $\alpha = (\alpha_0, ..., \alpha_{d+1}) \in \mathbb{B}^{d+2}$, partition the interval $[0,\pi]$ by subintervals
\begin{equation*}
W_j =	\begin{cases}
		[0, \alpha_0),												& \text{ if } j = 0, \\[1.5mm]
		[\alpha_0, \alpha_0+\alpha_1],									& \text{ if } j = 1, \\[1.5mm]
		\left( \sum \limits_{k=0}^{j-1} \alpha_k, \sum_{k=0}^j \alpha_k \right]	,	& \text{ if } j \in \{2,...,d+1\}.
		\end{cases}
\end{equation*}
We then define the partition $\mathcal{C}$ as
\begin{equation*}
\mathcal{C} = \left\{ C_j : j \in \{0, ..., d+1\} \right\},
\end{equation*}
where 
\begin{align*}
C_1 = \{0\} \cup \{ z \in \mathbb{H} : \operatorname{Arg} z \in W_1 \} \\
\intertext{and}
C_j = \{ z \in \ol{\mathbb{H}} : \operatorname{Arg} z \in W_j\} \text{ for $j \neq 1$}. \\
\end{align*}

\begin{figure}
\centering
\includegraphics[width=0.7\linewidth]{{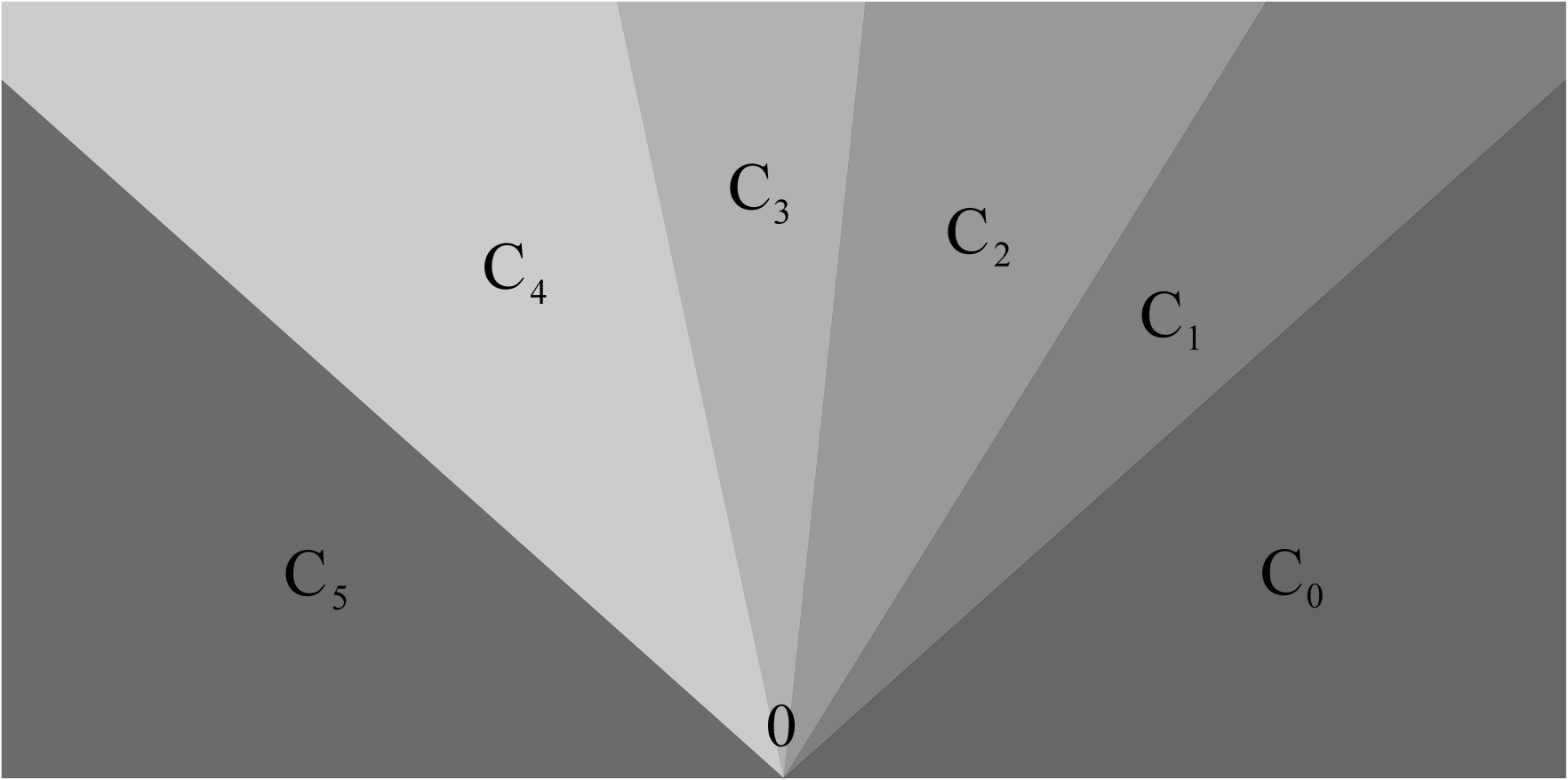}}
\caption{An example of a partition of the closed upper half plane $\overline{\mathbb{H}}$ into 6 cones.}
\label{P}
\end{figure}

The mapping $F_\kappa$ is defined as a composition
\begin{equation}
\label{F}
F_\kappa(z) = G \circ E(z),
\end{equation}
where $E$ is a permutation of the cones $C_1, ..., C_d$, $G$ is a piecewise horizontal translation, and $\kappa$ is a tuple of the parameters.  Formally, let $\tau$ be a permutation of $\{1,...,d\}$, that is a bijection $\tau: \{1,...,d\} \rightarrow \{1,...,d\}$, and let
\begin{equation*}
\theta_j(\alpha,\tau) = \sum_{\tau(k) < \tau(j)} \alpha_k - \sum_{k < j} \alpha_k.
\end{equation*}
When $\alpha$ and $\tau$ are unambiguous, we may refer to $\theta_j(\alpha,\tau)$ simply as $\theta_j$.  The map $E: \ol{\mathbb{H}} \rightarrow \ol{\mathbb{H}}$ is then defined as
\begin{equation*}
E(z) = 	\begin{cases}
		z			& \text{ if } z \in C_0 \cup C_{d+1} \\
		ze^{i\theta_j}	& \text{ if } z \in C_j, j \in \{1, ..., d\}
		\end{cases}.
\end{equation*}

\begin{figure}
\centering
\includegraphics[width=0.7\linewidth]{{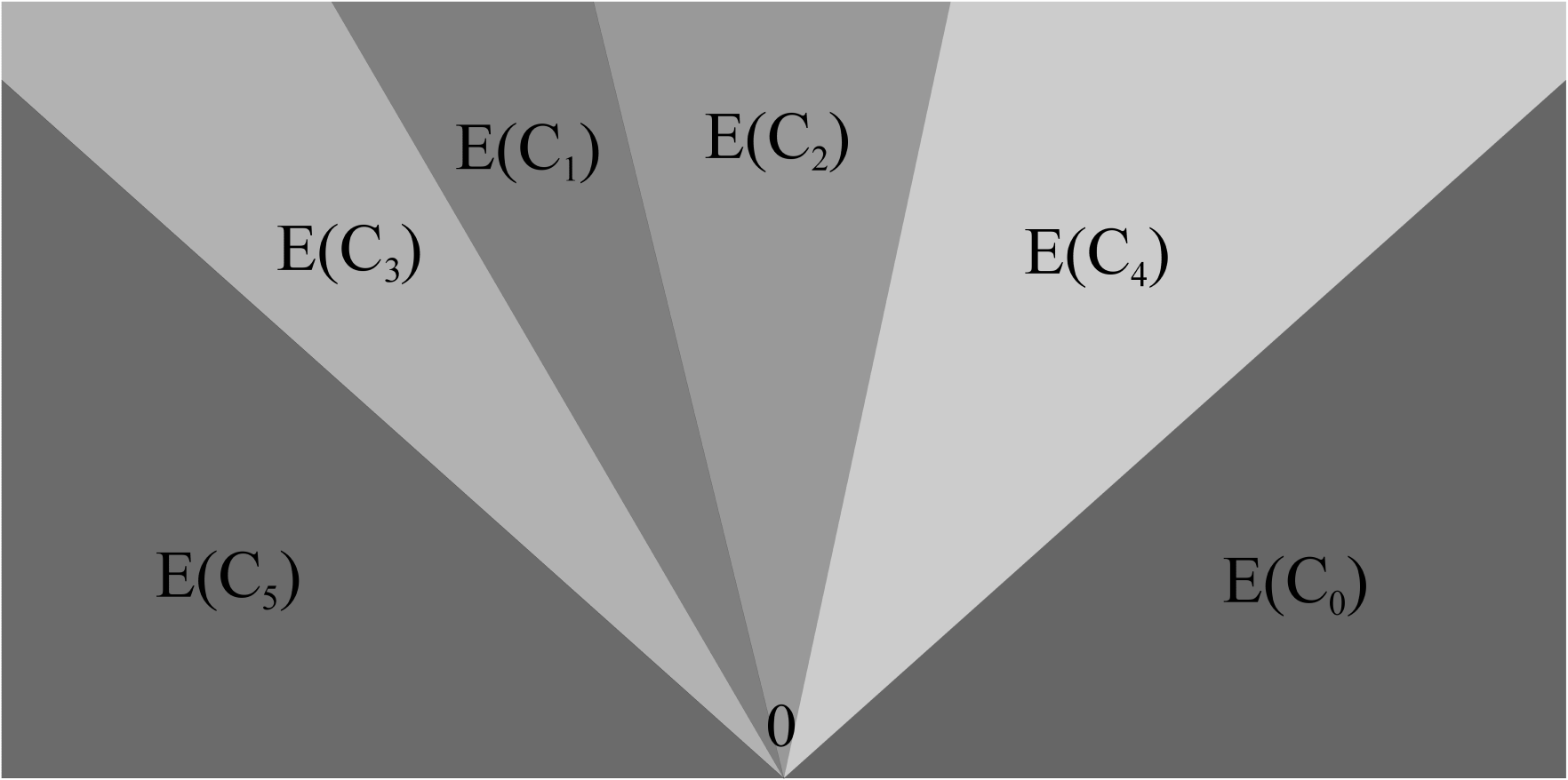}}
\caption{The same partition from figure \ref{P} after applying the cone exchange map $E$.}
\end{figure}

Note that $E$ is invertible Lebesgue-almost everywhere in $\ol{\mathbb{H}}$.  We define the \textit{middle cone} $C_c$ of $F_\kappa$ as
\begin{equation*}
C_c = \bigcup_{j=1}^d C_j = \mathbb{H} \setminus (C_0 \cup C_{d+1}).
\end{equation*}
The map $G: \ol{\mathbb{H}} \rightarrow \ol{\mathbb{H}}$ is defined as
\begin{equation}
\label{G}
G(z) = 	\begin{cases}
		z + \lambda	& \text{ if } z \in C_{d+1}, \\
		z - \eta		& \text{ if } z \in C_c, \\
		z - \rho		& \text{ if } z \in C_0,
		\end{cases}
\end{equation}
where $\rho, \lambda \in (0, \infty)$ are rationally independent, i.e. $\lambda /\rho \in \R \setminus \Q$, and $-\lambda < \eta < \rho$. Finally, we collect the parameters into the tuple $\kappa = (\alpha, \tau, \lambda, \eta, \rho)$.  See figures \ref{AsymmPlot1}, \ref{AsymmPlot2} and \ref{plot} for plots of the orbit structure for $F_\kappa$ for example choices of parameters $\kappa$.

\begin{figure}
\centering
\includegraphics[width=0.7\linewidth]{{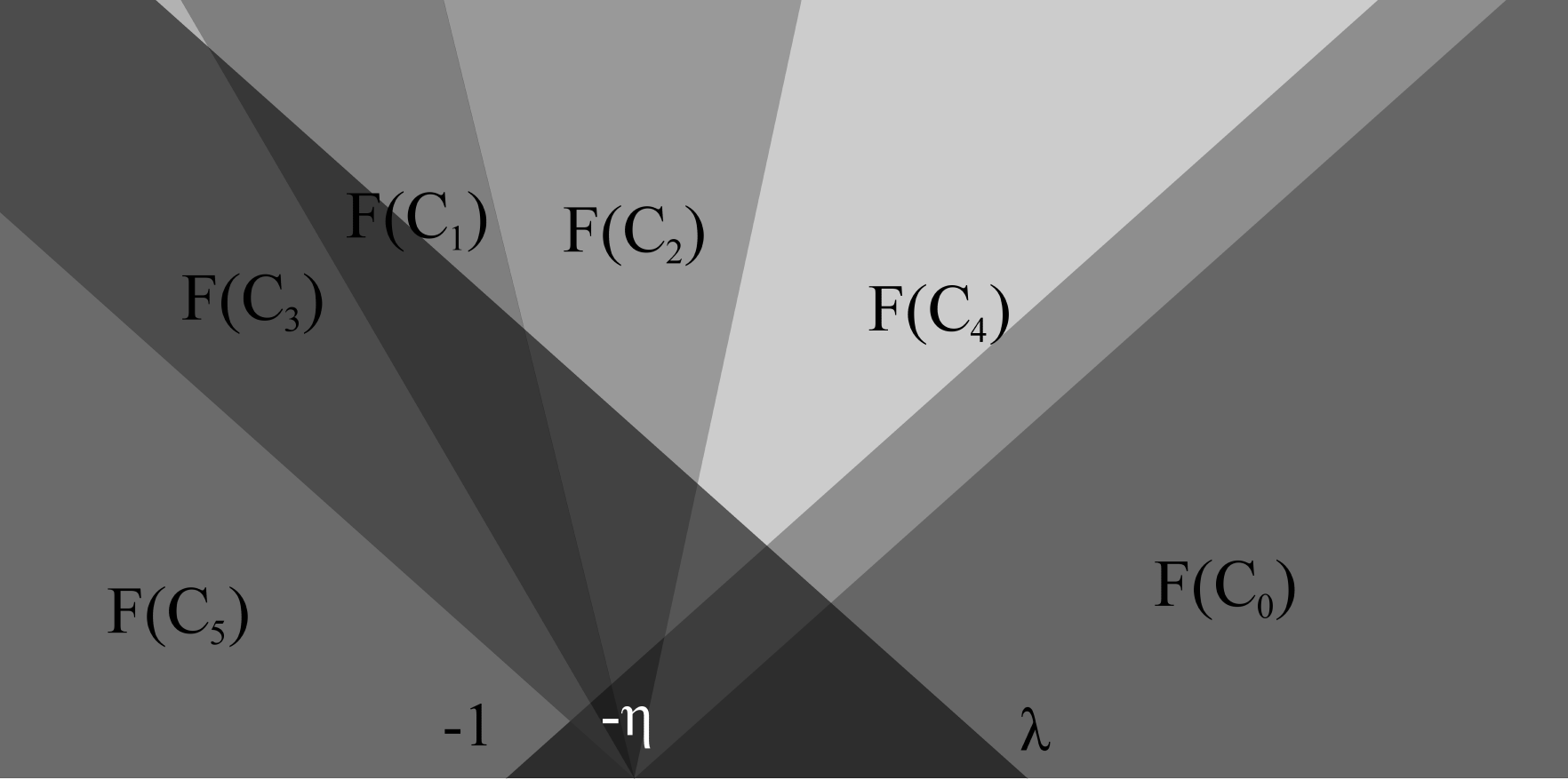}}
\caption{The example partition from figure \ref{P} after the TCE $F = G \circ E$ is performed.  Note the overlapping of the cones in a region containing 0.}
\end{figure}

For $a > 0$, let $F_{\kappa'}$ denote the TCE with parameters $\kappa' = (\alpha, \tau, \lambda/a, \eta/a, \rho/a)$.  Define $s_a : \ol{\mathbb{H}} \rightarrow \ol{\mathbb{H}}$ as uniform scaling by $a$ about the origin
\begin{equation}
\label{Sc}
s_a(z) = az.
\end{equation}

\begin{proposition}
\label{scale}
We have the following conjugacy:
\begin{equation}
\label{scalingconj}
F_{\kappa'}(z) = s_a^{-1} \circ F_\kappa \circ s_a(z).
\end{equation}
\end{proposition}
\begin{proof}
Firstly, observe that for all $j \in \{0,...,d+1\}$, 
\begin{equation*}
aC_j = C_j,
\end{equation*}
from which we can deduce that
\begin{equation}
\label{scalezP}
az \in C_j \text{ if and only if } z \in C_j.  
\end{equation}

Let $z \in \ol{\mathbb{H}}$.  From \eqref{Sc}, we get
\begin{equation*}
s_a^{-1} \circ F_\kappa \circ s_a(z) = \frac{1}{a}F_\kappa(az),
\end{equation*}
and by expanding $F_\kappa$ as in \eqref{F}, we have
\begin{equation*}
s_a^{-1} \circ F_\kappa \circ s_a(z) = 	\begin{cases}
						\frac{1}{a}(az + \lambda)			&\text{ if } az \in C_{d+1} \\
						\frac{1}{a}(e^{i\theta_j}(az) - \eta) 	&\text{ if } az \in C_j, j \in \{1,...,d\} \\
						\frac{1}{a}(az - \rho)				&\text{ if } az \in C_0
						\end{cases}.
\end{equation*}
Recalling \eqref{scalezP}, distributing the multiplication by $1/a$, we finally see that
\begin{align*}
s_a^{-1} \circ F_\kappa \circ s_a(z) 	&=	\begin{cases}
							z + \lambda/a			&\text{ if } z \in C_{d+1} \\
							e^{i\theta_j}z - \eta/a 	&\text{ if } z \in C_j, j \in \{1,...,d\} \\
							z - \rho/a				&\text{ if } z \in C_0
							\end{cases} \\
						&= F_{\kappa'}(z).
\end{align*}
\end{proof}

Clearly from this proposition, we can normalize $\rho = 1$ without any loss of generality.  Indeed, normalizing in this way proves very helpful for establishing recurrence results, as we shall see.

\begin{figure}
\centering
\includegraphics[width=\linewidth]{{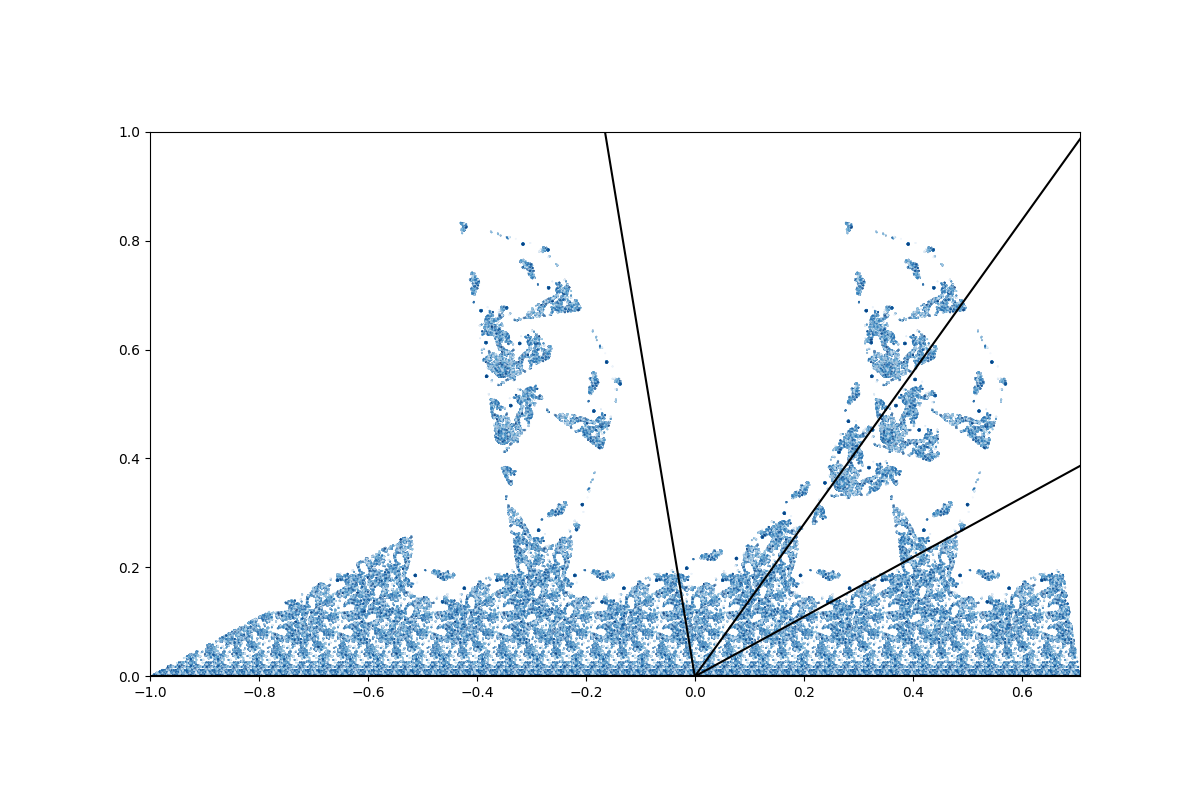}}
\caption{A plot of the first 3000 elements of the forward orbits of 1000 points (omitting the first 1500 to remove transients) chosen uniformly in the box $[-\rho,\lambda] \times [0,1]$ under a TCE with parameters $\alpha = (0.5,\pi/7,\pi/4,17\pi/28-0.5)$, $\tau: 1 \mapsto 2, 2 \mapsto 1$, $\lambda = \sqrt{2}/2$, $\eta = 1 - \lambda$ and $\rho = 1$.  Each orbit is given a (non-unique) colour to illustrate the trajectories.}
\label{AsymmPlot1}
\end{figure}

Simulations of the orbits of points under some TCEs appears to reveal complex behaviour even at small scales close to the real line, such as in figures \ref{AsymmPlot1} and \ref{AsymmPlot2}.  One way to investigate this behaviour  is by applying the tools of renormalization.  In particular, let $h: C_c \rightarrow \N \setminus \{0\}$ denote the \textit{first return time} of $z \in C_c$ to $C_c$ under $F_\kappa$, that is
\begin{equation}
\label{hz}
h(z) = \inf \left\{ n>0 : F_\kappa^n(z) \in C_c \right\}.
\end{equation}
The \textit{first return map} $R_\kappa : C_c \rightarrow C_c$ of $F_\kappa$ to $C_c$ is then defined as
\begin{equation}
\label{Rz}
R_\kappa(z) = F_\kappa^{h(z)}(z).
\end{equation}
Observe that for all $z \in C_c$, $R_\kappa(z) = F_\kappa^{h(z)}(z) = G^{h(z)} \circ E(z)$, since $E$ is the identity outside of $C_c$.

We will now state the main theorem of our paper in a simplified form, which we shall restate in more detail later as Theorem \ref{Renorm}, after establishing some terminology and preliminary results.

\begin{theorem}
Let $\alpha \in \mathbb{B}^{d+2}$, $\tau: \{1,...,d\} \rightarrow \{1,...,d\}$ be a bijection, $\lambda \in [0,1) \setminus \Q$, $-\lambda < \eta = p - q\lambda < 1$ for some $p,q \in \N \setminus \{0\}$ and set $\kappa = (\alpha,\tau,\lambda,\eta,1)$. Then there exist $\lambda' \in [0,1) \setminus \Q$, $\eta' \in \R$ of the form $-\lambda' < \eta' = p' - q'\lambda' < 1$ for some $p', q' \in \N \setminus\{0\}$, and a convex, positive area set $V \subset C_c$ containing 0 such that
\begin{equation*}
R_{\kappa'}(z) = \frac{1}{1 - \lambda_1\lambda} R_\kappa((1 - \lambda_1\lambda)z),
\end{equation*}
for all $z \in V$, where $\kappa' = (\alpha, \tau, \lambda', \eta', 1)$.
\end{theorem}

\begin{figure}
    \centering
    \includegraphics[width=\linewidth]{{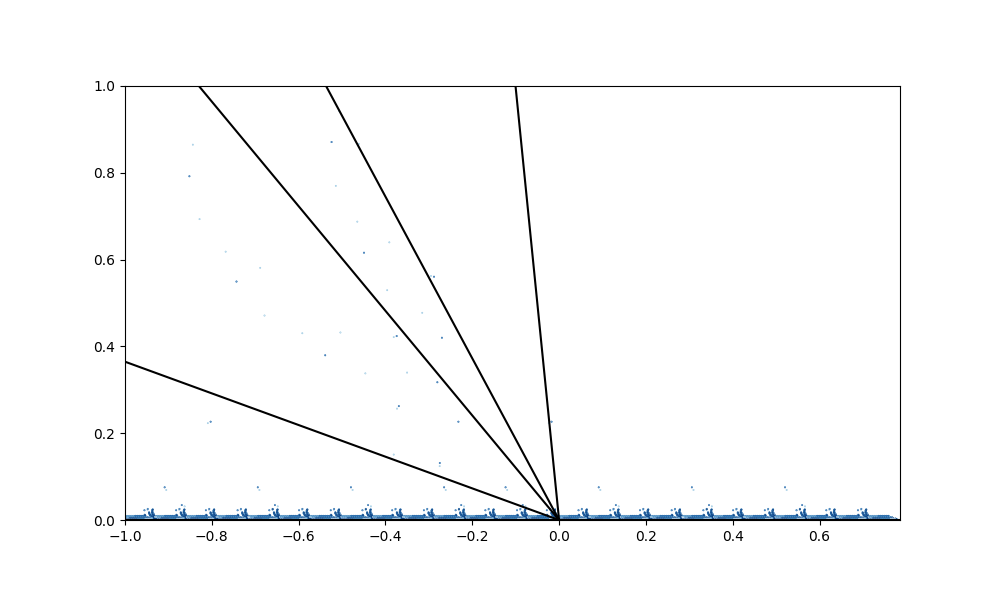}}
    \caption{A similar plot to figure \ref{AsymmPlot1}, this time of the first 3000 elements (excluding the first 1500 for transients) of the forward orbits of 1500 points chosen uniformly in the box $[-\rho,\lambda] \times [0,1]$ under a TCE with parameters $\alpha = (\pi/2+0.1,\pi/8,0.2,\pi/5-0.1,7\pi/40-0.2)$, $\tau: 1 \mapsto 3, 2 \mapsto 2, 3 \mapsto 1$, $\lambda = \pi/4$, and $\eta = 1 - \lambda$.  Note that the phase space is very sparse compared to figures \ref{AsymmPlot1} and \ref{plot}.}
    \label{AsymmPlot2}
\end{figure}

%%%%%%%%%%%%%%%%%%%%%%%%%%%%%%%%%%

\section{Tools}

In this section we prove some preliminary results that will serve as tools for more detailed investigation of the renormalization of TCEs.  Firstly, we note the following smaller observations.

\begin{proposition}
\label{FisG}
Let $F_\kappa$ be as in \eqref{F} and $G$ as in \eqref{G}.  Then for all $x \in \R$, 
\begin{equation*}
F_\kappa(x) = G(x).
\end{equation*}
\end{proposition}
\begin{proof}
This is clear from the observation that $E$ is the identity on $\R$.
\end{proof}

\begin{proposition}
The dynamics of $F_\kappa$ on $\R$ and on $\mathbb{H}$ are separate in the sense that $F_\kappa(\R) \subset \R$ and $F_\kappa(\mathbb{H}) \subset \mathbb{H}$.
\end{proposition}
\begin{proof}
The statement $F_\kappa(\R) \subset \R$ is clear from Proposition \ref{FisG} and the fact that $G$ is a horizontal translation. 

We now prove that $F_\kappa(\mathbb{H}) \subset \mathbb{H}$.  Suppose that $z \in \mathbb{H}$ such that $F_\kappa(z) \in \R$. Then $E(z) \in \R$, as $G$ is a horizontal translation.  If $z \in \mathbb{H} \setminus C_c$ then $\R \ni E(z) = z \in \mathbb{H}$, which is a contradiction. On the other hand, if $z \in C_c$, then $E(z) \in C_c$  But $\R \cap C_c = \{0\}$, meaning that $E(z) \in \R$ implies $E(z) = 0$.  But this is only the case if $z = 0 \in \R$, which is a contradiction.
\end{proof}

Let $\iota : \ol{\mathbb{H}} \rightarrow \{-1,0,1\}^\N$ denote the \textit{itinerary} for $F_\kappa$, defined by
\begin{equation*}
\iota(z) = \iota_0 \iota_1 \iota_2 ...,
\end{equation*}
where
\begin{equation*}
\iota_n = \iota_n(z) =  \begin{cases}
                        -1,     &\text{ if } F_\kappa^n(z) \in C_{d+1}, \\
                        0,      &\text{ if } F_\kappa^n(z) \in C_c, \\
                        1,      &\text{ if } F_\kappa^n(z) \in C_0.
                        \end{cases}     
\end{equation*}
This map is similar to, but distinct from the true notion of the \textit{encoding map}, since here we do not distinguish between the cones $C_1$, $C_2$, ..., $C_d$ within the middle cone $C_c$.  The next Lemma provides a crucial tool in the proof of Theorem \ref{conje}, since the dynamics on the interval $[-\rho,\lambda)$ is that of a rotation of the circle (except at the point $0$, in which case $F_\kappa(0) = -\eta$), which is more easily understood than that of an arbitrary point in $C_c$.

\begin{lemma}
\label{Split}
Let $\alpha \in \mathbb{B}^{d+2}$, $\tau: \{1,...,d\} \rightarrow \{1,...,d\}$ be a bijection, $\lambda, \rho \in \R$ such that $\lambda/\rho \notin \mathbb{Q}$, and $-\lambda < \eta < 1$.  If $z \in C_c$, then for all $1 \leq j \leq h(z)$,
\begin{equation*}
F_\kappa^j(z) = E(z) + F_\kappa^j(0).
\end{equation*}
\end{lemma}
\begin{proof}
Suppose not, for a contradiction.  Then there is some $n$ with $0 \leq n \leq h(z)$ such that $F_\kappa^n(z) \neq E(z) + F_\kappa^n(0)$, and without loss of generality assume that $n$ is the smallest such integer.  Clearly $n > 1$, so for all $0 \leq j \leq n-1$, $F_\kappa^j(z) - E(z) = F_\kappa^j(0)$, and therefore $\iota_j(z) = \iota_j(0)$ for all $0 \leq j \leq n-2$, but $\iota_{n-1}(z) \neq \iota_{n-1}(0)$.  Since $n \leq h(z)$, we cannot have $F_\kappa^j(z) \in C_c$ for all $1 \leq j \leq n-1$.  Hence for addresses of the $(n-1)^{\text{th}}$ iterates of $z$ and $0$ to disagree, one of two cases must occur:
\begin{enumerate}[label={\arabic*.}]
\item $F_\kappa^{n-1}(z) \in C_{d+1}$ and $F_\kappa^{n-1}(0) \in C_c \cup C_0$; or \\
\item $F_\kappa^{n-1}(z) \in C_{0}$ and $F_\kappa^{n-1}(0) \in C_c \cup C_{d+1}$.
\end{enumerate}
Since the orbit of $0$ is restricted to $\R$ and since $E(0) = 0$, the second parts of each case become $G^{n-1}(0) \geq 0$ and $G^{n-1}(0) \leq 0$, respectively.

Suppose $E(z) = z'$, and let $\varepsilon_1 = \Im(z') \operatorname{cot}(\alpha_{d+1})$ and $\varepsilon_2 = \Im(z') \operatorname{cot}(\alpha_{0})$.  Then $z' \in C_c$ if and only if $-\varepsilon_1 \leq \Re(z') \leq \varepsilon_2$.  Note that since $G$ is a horizontal translation, $F_\kappa^{n-1}(z) = G^{n-1}(E(z)) = G^{n-1}(z') \in C_{d+1}$ if and only if $\Re(G^{n-1}(z')) < -\varepsilon_1$.  Similarly $G^{n-1}(z') \in C_{0}$ if and only if $\Re(G^{n-1}(z')) > \varepsilon_2$. The two above cases above can thus be reformulated as:
\begin{enumerate}[label={\arabic*.}]
\item $\Re(G^{n-1}(z')) < -\varepsilon_1$ and $G^{n-1}(0) \geq 0$; or \\
\item $\Re(G^{n-1}(z')) > \varepsilon_2$ and $G^{n-1}(0) \leq 0$.
\end{enumerate}

In case 1, we get $0 \leq G^{n-1}(0) = G^{n-1}(z') - z' = \Re(G^{n-1}(z')) - \Re(z')$.  Hence $\Re(G^{n-1}(z)) \geq \Re(z')$ and thus
\begin{equation*}
-\varepsilon_1 \leq \Re(z') \leq \Re(G^{n-1}(z')) < -\varepsilon_1,
\end{equation*}
which is a contradiction.  Case 2 leads to a similar contradiction.  Therefore there is no such $n$.
\end{proof}

\subsection{Continued Fractions}

Recall from the theory of continued fractions that the $n^{\text{th}}$ convergent to a positive, irrational real number $\lambda = [\lambda_0;\lambda_1,\lambda_2,...]$ is a fraction $p_n/q_n = [\lambda_0;\lambda_1,...,\lambda_n]$, where $p_n,q_n$ are coprime integers and $q_n > 0$.  The numbers $p_n,q_n$ can be generated by the recursive relations:
\begin{equation}
\label{convergents}
\begin{aligned}
p_0 &= \lambda_0,				&&q_0 = 1, \\
p_1 &= \lambda_1\lambda_0 + 1,			&&q_1 = \lambda_1, \\
p_n &= \lambda_np_{n-1} + p_{n-2}, 	&&q_n = \lambda_nq_{n-1} + q_{n-2}.
\end{aligned}
\end{equation}
Furthermore, the convergents to $\lambda$ satisfy the property that for all positive integers $s < q_{n+1}$ and all $r \in \Z$,
\begin{equation*}
| q_n\lambda - p_n | \leq | s\lambda - r |,
\end{equation*}
with equality only when $(r,s) = (p_n,q_n)$.  Also observe that we can use the recurrence relation in \eqref{convergents} to set
\begin{equation}
\label{negative}
\begin{aligned}
p_{-1} &= 1          &&q_{-1} = 0.
\end{aligned}
\end{equation}

Let $g: [0,1] \rightarrow [0,1]$ denote the Gauss map, given by
\begin{equation*}
g(x) = \frac{1}{x} - \left\lfloor \frac{1}{x} \right\rfloor.
\end{equation*}
In particular, if $\lambda = [0; \lambda_1, \lambda _2, \lambda_3, ...] \in [0,1]$, then
\begin{equation*}
g(\lambda) = \frac{1}{\lambda} - \lambda_1 = [0; \lambda_2, \lambda_3, ...].
\end{equation*}
Let $\lambda = [0; \lambda_1, \lambda_2, ...] \in [0,1) \setminus \mathbb{Q}$.  To start, let
\begin{equation*}
\mathcal{N}_\lambda = \{(m,n) \in \N^2 : 0 \leq n \leq \lambda_{m+1}\},
\end{equation*}
and define the function $w_\lambda : \mathcal{N}_\lambda \rightarrow \N$ by
\begin{equation*}
w_\lambda(m,n) = 	\begin{cases}
				n 							& \text{ if } m = 0, \\
				\lambda_1 + ... + \lambda_m + n 	& \text{ if } m > 0.
				\end{cases}
\end{equation*}
Note that $w_\lambda$ is surjective and in fact
\begin{equation}
\label{wrapping}
w_\lambda(m+1,0) = w_\lambda(m,\lambda_{m+1}).
\end{equation}
Furthermore, if we define the subset $\mathcal{N}_\lambda^< \subset \mathcal{N}_\lambda$ to be
\begin{equation*}
\mathcal{N}_\lambda^< =\{(m,n) \in \N^2 : 0 \leq n < \lambda_{m+1} \},
\end{equation*}
then $w_\lambda |_{\mathcal{N}_\lambda^<}$ is a bijection.

From now on, we denote the $j^{\text{th}}$ coefficient of the continued fraction expansion of $g^m(\lambda)$ by $g^m(\lambda)_j$.  The next proposition gives us a nice relationship between the set $\mathcal{N}_\lambda$ and the Gauss map $g$.

\begin{proposition}
\label{interchange}
Let $j,m,n \in \N$.  Then $(m+j,n) \in \mathcal{N}_\lambda$ if and only if $(j,n) \in \mathcal{N}_{g^m(\lambda)}$.  Moreover,
\begin{equation*}
w_\lambda(m+j,n) = \lambda_1 + ... + \lambda_m + w_{g^m(\lambda)}(j,n).
\end{equation*}
\end{proposition}
\begin{proof}
We have that $(m+j,n) \in \mathcal{N}_\lambda$ is equivalent to $0 \leq n \leq \lambda_{m+j+1}$.  We also have that $\lambda_{m+j+1} = g^m(\lambda)_{j+1}$, so $0 \leq n \leq g^m(\lambda)_{j+1}$.  This is equivalent to $(j,n) \in \mathcal{N}_{g^m(\lambda)}$.

If $m=j=0$, then the second part of our lemma is clearly true.  

Assume $j = 0$ and $m > 0$.  Then
\begin{equation*}
w_\lambda(m+j,n) = w_\lambda(m,n) = \lambda_1 + ... + \lambda_m + n = \lambda_1 + ... + \lambda_m + w_{g^m(\lambda)}(0,n),
\end{equation*}
where the final equality is true since $(m,n) \in \mathcal{N}_\lambda$ is equivalent to $(0,n) \in \mathcal{N}_{g^m(\lambda)}$. 

Finally, suppose $m,j > 0$.  Then
\begin{align*}
w_\lambda(m+j,n) 	&= \lambda_1 + ... + \lambda_m + \lambda_{m+1} + ... + \lambda_{m+j} + n \\
			&= \lambda_1 + ... + \lambda_m + g^m(\lambda)_1 + ... + g^m(\lambda)_j + n.
\end{align*}
Note that since $(m+j,n) \in \mathcal{N}_\lambda$ is equivalent to $(j,n) \in \mathcal{N}_{g^m(\lambda)}$, we have
\begin{equation*}
w_\lambda(m+j,n) = \lambda_1 + ... + \lambda_m + w_{g^m(\lambda)}(j,n).
\end{equation*}
\end{proof}
The bijection $w_\lambda$ mainly serves as a way to show that there is a ``natural'' well-ordering for the set $\mathcal{N}_\lambda^<$, which allows us to meaningfully index sequences by $\mathcal{N}_\lambda^<$ and, as we shall see later, define the notion of a maximal element of a finite subset of $\mathcal{N}_\lambda$.

We define the \textit{one-sided convergents} (or \textit{semiconvergents}) to $\lambda$ as the fractions
\begin{equation*}
\frac{P_{m,n}(\lambda)}{Q_{m,n}(\lambda)} = \begin{cases}
[0; \lambda_1,...,\lambda_m]    &\text{ if } n = 0, \\
[0; \lambda_1,...,\lambda_m,n]  &\text{ if } n > 0,
\end{cases}
\end{equation*}
Indeed, this formula is compatible with the indexing $w_\lambda$ in that
\begin{equation*}
\frac{P_{m,0}(\lambda)}{Q_{m,0}(\lambda)} = \frac{P_{m-1,\lambda_m}(\lambda)}{Q_{m-1,\lambda_m}(\lambda)}.
\end{equation*}

A standard result in the theory of continued fractions is that for $(m,n) \in \mathcal{N}_\lambda^<$, we have
\begin{equation}
\label{onesideconv}
\frac{P_{m,n}(\lambda)}{Q_{m,n}(\lambda)} = \begin{cases} 
    \frac{np_m(\lambda) + p_{m-1}(\lambda)}{nq_m(\lambda) + q_{m-1}(\lambda)},  &\text{ if } n \neq 0 \\[2mm]
	\frac{p_m(\lambda)}{q_m(\lambda)} &\text{ if } n = 0.
\end{cases}
\end{equation}
One way to interpret these fractions is as being the best rational approximates of $\lambda$ from one ``direction''.  In particular, borrowing notation from the beginning of section 2.2 in \cite{PR18}, we have
\begin{align*}
\left( \frac{P_{m,n}}{Q_{m,n}} \right)_{(m,n) \in \mathcal{N}_\lambda^<, m \text{ even}}  &= \left( \frac{p_k'}{q_k'} \right)_{k \in \N}, \text{ and} \\
\left( \frac{P_{m,n}}{Q_{m,n}} \right)_{(m,n) \in \mathcal{N}_\lambda^<, m \text{ odd}}   &= \left( \frac{p_k''}{q_k''} \right)_{k \in \N},
\end{align*}
where $(p_k'/q_k')_k$ are the \textit{best rational approximates from above} in the sense that $p_k'/q_k' > \lambda$ and for all rational numbers $r/s \neq p_k'/q_k'$ such that $r/s > \lambda$ and $1 \leq s < q_{k+1}'$, we have
\begin{equation}
\label{errorabove}
0 < |q_k'\lambda - p_k'| < |s\lambda - r|,
\end{equation}
and in a similar fashion $(p_k''/q_k'')_k$ are the \textit{best rational approximates from below} in the same sense except that $p_k''/q_k'' < \lambda$ and
\begin{equation}
\label{errorbelow}
0 < |q_k''\lambda - p_k''| < |s\lambda - r|,
\end{equation}
holds for $r/s \neq p_k''/q_k''$ such that $r/s < \lambda$ and $1 \leq s < q_{k+1}''$.

We define the sequence $(\Delta_{m,n}(\lambda))_{(m,n) \in \mathcal{N}_\lambda}$ by
\begin{equation}
\label{semiconverrors}
\Delta_{m,n}(\lambda) = Q_{m,n}(\lambda)\lambda - P_{m,n}(\lambda).
\end{equation}
We see immediately from the above discussion that $P_{m,n}(\lambda)/Q_{m,n}(\lambda) < \lambda$ if and only if $m$ is odd and $n > 0$ or $m$ is even and $n = 0$, that is
\begin{equation}
\label{oddpositive}
\Delta_{m,n}(\lambda) > 0 \text{ if and only if } \begin{cases}
\text{$m$ is even and $n = 0$, or} \\
\text{$m$ is odd and $n > 0$.}
\end{cases}
\end{equation}

By expanding the definitions of $P_{m,n}$ and $Q_{m,n}$, we see
\begin{equation}
\label{Deltan}
\begin{aligned}
\Delta_{m,n}(\lambda) 	&= n(q_m\lambda - p_m) + q_{m-1}\lambda - p_{m-1} \\
							&= n\Delta_{m,0}(\lambda) + \Delta_{m-1,0}(\lambda),
\end{aligned}
\end{equation}
for $(m,n) \in \mathcal{N}_\lambda$ with $m \geq 1$.  Moreover, by expanding the recurrence relation for $p_m$ and $q_m$ and rearranging terms, we have the additional property
\begin{align*}
\Delta_{m,0}(\lambda) 	&= q_m\lambda - p_m \\
							&= \lambda_m(q_{m-1}\lambda - p_{m-1}) + q_{m-2}\lambda - p_{m-2} \\
							&=\lambda_m\Delta_{m-1,0}(\lambda) + \Delta_{m-2,0}(\lambda),
\end{align*}
for $m \in \N, m \geq 2$.

Note that in agreement with the function $w_\lambda$, we have $\Delta_{m-1,\lambda_m}(\lambda) = \Delta_{m,0}(\lambda)$.  A result by Bates et al. \cite{BBT05} presents an interesting connection between iterates of the Gauss map and consecutive errors in the approximation of $\lambda$ by its convergents.

\begin{lemma}[Theorem 10 of \cite{BBT05}]
Let $\lambda = [0; \lambda_1, \lambda_2, ...] \in [0,1) \setminus \mathbb{Q}$.  For all $m \in \N$,
\begin{equation}
\label{Giterate}
g^m(\lambda) = \frac{q_m\lambda - p_m}{p_{m-1} - q_{m-1}\lambda}.
\end{equation}
\end{lemma}

Equation \eqref{Giterate} can be equivalently formulated as
\begin{equation}
\label{gerror}
g^m(\lambda) = -\frac{\Delta_{m,0}(\lambda)}{\Delta_{m-1,0}(\lambda)}.
\end{equation}

\begin{lemma}
\label{fractions}
Let $\lambda = [0; \lambda_1, \lambda_2, ...] \in [0,1) \setminus \mathbb{Q}$.  For all $(m,n) \in \mathcal{N}_\lambda$ with $m \geq 1$,
\begin{equation}
\label{GN}
\Delta_{0,n}(g^m(\lambda)) = -\frac{\Delta_{m,n}(\lambda)}{\Delta_{m-1,0}(\lambda)}.
\end{equation}

\end{lemma}

\begin{proof}
For $n=0$, $\Delta_{0,0}(g^m(\lambda)) = g^m(\lambda)$, so \eqref{GN} holds.

Next, observe that
\begin{equation*}
\Delta_{j,0}(g^m(\lambda)) = g^m(\lambda)_j\Delta_{j-1,0}(g^m(\lambda)) + \Delta_{j-2,0}(g^m(\lambda)).
\end{equation*}
Since $g^m(\lambda)_j = \lambda_{m+j}$, this becomes
\begin{equation*}
\Delta_{j,0}(g^m(\lambda)) = \lambda_{m+j}\Delta_{j-1,0}(g^m(\lambda)) + \Delta_{j-2,0}(g^m(\lambda)).
\end{equation*}
We thus see that
\begin{equation*}
\Delta_{0,n}(g^m(\lambda)) = n(q_0g^m(\lambda) - p_0) + q_{-1}g^m(\lambda) - p_{-1},
\end{equation*}
and by recalling $p_{-1}$, $q_{-1}$, $p_0$, and $q_0$ from \eqref{convergents}, we get
\begin{equation*}
\Delta_{0,n}(g^m(\lambda)) = ng^m(\lambda) -1.
\end{equation*}
Using \eqref{gerror}, we can substitute $g^m(\lambda)$ and rearrange terms to get
\begin{align*}
\Delta_{0,n}(g^m(\lambda)) 	&= -n\left(\frac{\Delta_{m,0}(\lambda)}{\Delta_{m-1,0}(\lambda)} + 1 \right) \\
						&= -\frac{n\Delta_{m,0}(\lambda) + \Delta_{m-1,0}(\lambda)}{\Delta_{m-1,0}(\lambda)}.
\end{align*}
Finally, from this and \eqref{Deltan} we get \eqref{GN}.
\end{proof}

\begin{corollary}
For all $(m+j,n) \in \mathcal{N}_\lambda$, where $m,j \in \N$, $m \geq 1$, we have
\begin{equation}
\label{gchange}
\Delta_{0,n}(g^{m+j}(\lambda)) = -\frac{\Delta_{m,n}(g^j(\lambda))}{\Delta_{m-1,0}(g^j(\lambda))}
\end{equation}
\end{corollary}

\begin{proof}
This follows from Lemma \ref{fractions} with $g^j(\lambda)$ instead of $\lambda$.
\end{proof}

Our next Lemma is an important tool for determining scaling properties of these errors.

\begin{lemma}
\label{Deltas}
Let $\lambda = [0; \lambda_1, \lambda_2, ...] \in [0,1) \setminus \mathbb{Q}$.  For all $(m+j,n) \in \mathcal{N}_\lambda$ such that $m, j \in \N$ and $m \geq 1$,
\begin{equation}
\label{Delta}
\Delta_{j,n}(g^m(\lambda)) = -\frac{\Delta_{m+j,n}(\lambda)}{\Delta_{m-1,0}(\lambda)}.
\end{equation}
\end{lemma}
\begin{proof}
Let us prove first that \eqref{Delta} holds for $n=0$.  By multiplying and dividing by $\Delta_{m+k-1,0}$ for all $0 \leq k \leq j$, we get
\begin{equation*}
\frac{\Delta_{m+j,0}(\lambda)}{\Delta_{m-1,0}(\lambda)} = \prod_{k=0}^j \frac{\Delta_{m+k,0}(\lambda)}{\Delta_{m+k-1,0}(\lambda)}
\end{equation*}
Then, using \eqref{gerror}, we get
\begin{equation*}
\frac{\Delta_{m+j,0}(\lambda)}{\Delta_{m-1,0}(\lambda)} = \prod_{k=0}^j -\Delta_{0,0}(g^{m+k}(\lambda)).
\end{equation*}
Rearranging this last expression and using  \eqref{gchange}, we get
\begin{equation*}
\frac{\Delta_{m+j,0}(\lambda)}{\Delta_{m-1,0}(\lambda)} = (-1)^{j+1} \Delta_{0,0}(g^m(\lambda)) \prod_{k=1}^j -\frac{\Delta_{k,0}(g^m(\lambda))}{\Delta_{k-1,0}(g^m(\lambda))}
\end{equation*}
We then simplify the product by cancelling terms in the numerator and denominator to get
\begin{align*}
\frac{\Delta_{m+j,0}(\lambda)}{\Delta_{m-1,0}(\lambda)} 	
	&= (-1)^{j+1} \Delta_{0,0}(g^m(\lambda)) \left( (-1)^j \frac{\Delta_{j,0}(g^m(\lambda))}{\Delta_{0,0}(g^m(\lambda))} \right) \\
	&= -\Delta_{j,0}(g^m(\lambda)).
\end{align*}
Finally, for general $(m+j,n) \in \mathcal{N}_\lambda$, $m,j \in \N$, $m \geq 1$, we have
\begin{equation*}
\frac{\Delta_{m+j,n}(\lambda)}{\Delta_{m-1,0}(\lambda)} = \frac{\Delta_{m+j,n}(\lambda)}{\Delta_{m+j-1,0}(\lambda)} \frac{\Delta_{m+j-1,0}(\lambda)}{\Delta_{m-1,0}(\lambda)}
\end{equation*}
Using \eqref{GN} and \eqref{Delta} for $n=0$, we get
\begin{equation*}
\frac{\Delta_{m+j,n}(\lambda)}{\Delta_{m-1,0}(\lambda)} = \Delta_{0,n}(g^{m+j}(\lambda)) \Delta_{j-1,0}(g^m(\lambda)),
\end{equation*}
and then using \eqref{gchange} gives us
\begin{align*}
\frac{\Delta_{m+j,n}(\lambda)}{\Delta_{m-1,0}(\lambda)} 
	&= -\frac{\Delta_{j,n}(g^m(\lambda))}{\Delta_{j-1,0}(g^m(\lambda))} \Delta_{j-1,0}(g^m(\lambda)) \\
	&= -\Delta_{j,n}(g^m(\lambda)),
\end{align*}
as required.
\end{proof}

With these properties in mind, we will now define the sets which will partition a neighbourhood of the middle cone $C_c$.  Recall that $\mathcal{N}_\lambda^<$ denotes the subset of $\mathcal{N}_\lambda$ defined by
\begin{equation*}
\mathcal{N}_\lambda^< = \{(m,n) \in \N^2 : 0 \leq n < \lambda_{m+1} \}.
\end{equation*}
For $(m,n) \in \mathcal{N}_\lambda^<$, let $S_{m,n}(\lambda)$ be the set defined by

\begin{equation}
\label{defSmn}
S_{m,n}(\lambda) =	
	\begin{cases}
        \begin{aligned}
			&(C_0 - \Delta_{m,0}(\lambda)) \cap C_c \cap (C_c - \Delta_{m,n+1}(\lambda)) \\
			&\cap  (C_{d+1} - (n\Delta_{m,0}(\lambda) + \Delta_{m-1,0}(\lambda)))
		\end{aligned},
			&\text{if } m \text{ is even}, \\[6mm]
		\begin{aligned}
			&(C_0 - (n\Delta_{m,0}(\lambda) + \Delta_{m-1,0}(\lambda))) \cap C_c \\
			&\cap (C_c - \Delta_{m,n+1}(\lambda)) \cap (C_{d+1} - \Delta_{m,0}(\lambda)) 	
		\end{aligned},
			&\text{if } m \text{ is odd}.
	\end{cases}
\end{equation}
For brevity, we will drop the argument $\lambda$ from $S_{m,n}(\lambda)$ if it is unambiguous.  Additionally for the purposes of the case that $m=0$, and recalling \eqref{negative}, we have
\begin{equation*}
\Delta_{-1,0}(\lambda) = q_{-1}\lambda - p_{-1} = -1.
\end{equation*}
Recall from \eqref{oddpositive} that for $(m,n) \in \mathcal{N}_\lambda^<$, $\Delta_{m,n} > 0$ if and only if $m$ is odd and $n > 0$ or $m$ is even and $n = 0$.  Thus we can clearly see that $S_{m,n} \neq \emptyset$ for all $(m,n) \in \mathcal{N}_\lambda^<$.  Additionally, since every point in $S_{m,n}$ has positive imaginary part, the boundary of $S_{m,n}$ consists of segments of the non-horizontal boundary lines of $C_0$, $C_c$, and $C_{d+1}$, and all of these lines either have angle $\alpha_0$ or $\pi - \alpha_{d+1}$.  Thus, $S_{m,n}$ is a quadrilateral, and its opposing sides must be parallel, so it is a parallelogram.

Indeed, since opposite edges of $S_{m,n}$ are parallel, the side lengths of $S_{m,n}$ are uniquely determined by the horizontal distances between opposing edges.  In the case that $m$ is even, these are precisely the distances between the vertices of the pairs of cones $C_0 - \Delta_{m,0}$ and $C_c$, and $C_c - \Delta_{m,n+1}$ and $C_{d+1} - \Delta_{m,n}$.  In the case that $m$ is odd, the horizontal distances are determined by the distance between the vertices of pairs of cones $C_0 - \Delta_{m,n}$ and $C_c - \Delta_{m,n+1}$, and $C_c$ and $C_{d+1} - \Delta_{m,0}$.  Since $\Delta_{m,n+1} - \Delta_{m,n} = \Delta_{m,0}$, we know that these distances are equal.  Therefore, the side lengths of opposing edges of $S_{m,n}$ are equal and can be calculated as
\begin{equation}
\label{sidelengths}
\frac{\Delta_{m,n} \sin\alpha_0}{\sin(\alpha_0 + \alpha_{d+1})} \text{ and } \frac{\Delta_{m,n} \sin\alpha_{d+1}}{\sin(\alpha_0 + \alpha_{d+1})}
\end{equation}
These sidelengths are equal only when $\alpha_0 = \alpha_{d+1}$, in which case $S_{m,n}$ is a rhombus for all $(m,n) \in \mathcal{N}_\lambda^<$.  See figure \ref{pconstruct} for an example of the geometry of the construction of sets $S_{m,n}$.

\begin{figure}
    \centering
    \includegraphics[width=\linewidth]{{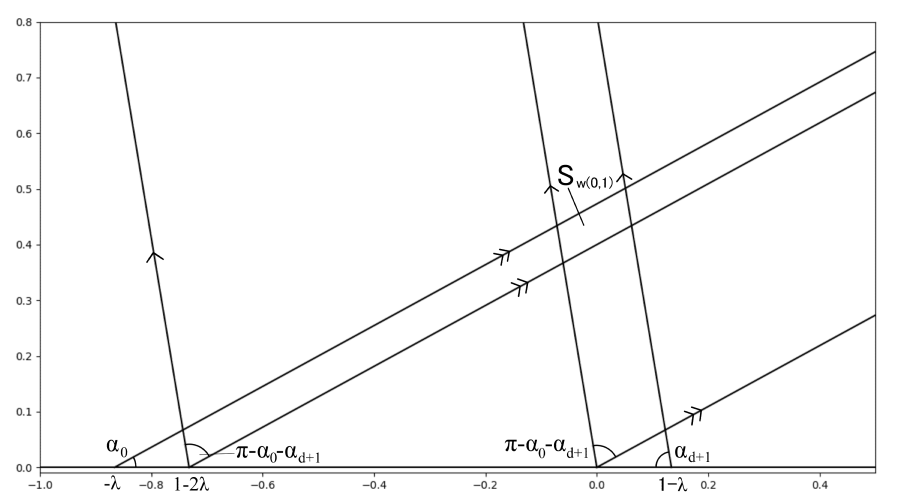}}
    \caption{An illustration of the construction of the parallelogram $S_{0,1}(\lambda)$ for the parameters in figure \ref{AsymmPlot1}.  Here, the angles shown indicate the cones used to construct $S_{1,0}(\lambda)$.  In this case, the vertices of these cones can be verified via \eqref{defSmn} to be $-\Delta_{0,0}(\lambda) = -\lambda$, $-\Delta_{1,1}(\lambda) = 1 - 2\lambda$, 0 and $-(\Delta_{0,0}(\lambda) + \Delta_{-1,0}) = 1 - \lambda$.}
    \label{pconstruct}
\end{figure}

An interesting property of these sets can be found by an application of Lemma \ref{Deltas}.

\begin{theorem}
\label{Scalingprop}
Let $\lambda \in [0,1) \setminus \mathbb{Q}$.  For all $(m+j,n) \in \mathcal{N}_\lambda^<$ such that $m,j \in \N$ and $m \geq 2$ is even,
\begin{equation*}
\frac{1}{|\Delta_{m-1,0}(\lambda)|}S_{m+j,n}(\lambda) = S_{j,n}(g^m(\lambda)).
\end{equation*}
\end{theorem}

\begin{proof}
Firstly, recall from \eqref{oddpositive} that since $m$ is even, $\Delta_{m-1,0}(\lambda) < 0$, and thus 
\begin{equation*}
\frac{1}{|\Delta_{m-1,0}(\lambda)|} = \frac{1}{-\Delta_{m-1,0}(\lambda)} > 0.
\end{equation*}
Hence,
\begin{equation*}
\frac{1}{|\Delta_{m-1,0}(\lambda)|}(C_k - x) = C_k + \frac{x}{\Delta_{m-1,0}(\lambda)},
\end{equation*}
for all $k \in \{1,...,d\}$ and all $x \in \R$.  Thus, we have
\begingroup
\allowdisplaybreaks
\begin{align*}
&\frac{1}{|\Delta_{m-1,0}(\lambda)|}S_{m+j,n}(\lambda) \\
&= 	\begin{cases}
	\begin{aligned}
	&\left( C_0 + \frac{ \Delta_{m+j,0}(\lambda) }{ \Delta_{m-1,0}(\lambda) } \right) \cap \left( C_c + \frac{ \Delta_{m+j,n+1}(\lambda) }{ \Delta_{m-1,0}(\lambda) } \right) \\[2mm]
	&\cap C_c \cap \left( C_{d+1} + \frac{ n\Delta_{m+j,0}(\lambda) + \Delta_{m+j-1,0}(\lambda) }{ \Delta_{m-1,0}(\lambda) } \right),
	\end{aligned}	&\text{if } m+j \text{ is even,} \\[11mm]
	\begin{aligned}
	&\left( C_0 + \frac{ n\Delta_{m+j,0}(\lambda) + \Delta_{m+j-1,0}(\lambda) }{ \Delta_{m-1,0}(\lambda) } \right) \cap C_c \\[2mm]
	&\cap \left( C_c + \frac{ \Delta_{m+j,n+1}(\lambda) }{ \Delta_{m-1,0}(\lambda) } \right) \cap \left( C_{d+1} + \frac{ \Delta_{m+j,0}(\lambda) }{ \Delta_{m-1,0}(\lambda) } \right),
	\end{aligned}	&\text{if } m+j \text{ is odd.}
	\end{cases}
\end{align*}
Using \eqref{Delta} many times give us
\begin{align*}
&\frac{1}{|\Delta_{m-1,0}(\lambda)|}S_{m+j,n}(\lambda) \\
&=	\begin{cases}
	\begin{aligned}
	&(C_0 - \Delta_{j,0}(g^m(\lambda))) \cap (C_c - \Delta_{j,n+1}(g^m(\lambda))) \\
	&\cap C_c \cap (C_{d+1} - (n\Delta_{j,0}(g^m(\lambda)) + \Delta_{j-1,0}(g^m(\lambda)))), 
	\end{aligned}	&\text{if } j \text{ is even,} \\[5mm]
	\begin{aligned}
	&(C_0 - (n\Delta_{j,0}(g^m(\lambda)) + \Delta_{j-1,0}(g^m(\lambda)))) \cap C_c \\
	&\cap (C_c - \Delta_{j,n+1}(g^m(\lambda))) \cap (C_{d+1} - \Delta_{j,0}(g^m(\lambda))),
	\end{aligned}	&\text{if } j \text{ is odd.}
	\end{cases}
\end{align*}
Finally, comparing this with \eqref{defSmn} gets
\begin{equation*}
\frac{1}{|\Delta_{m-1,0}(\lambda)|}S_{m+j,n}(\lambda) = S_{j,n}(g^m(\lambda)).
\end{equation*}
\endgroup
\end{proof}

\begin{figure}
    \centering
    \includegraphics[width=\linewidth]{{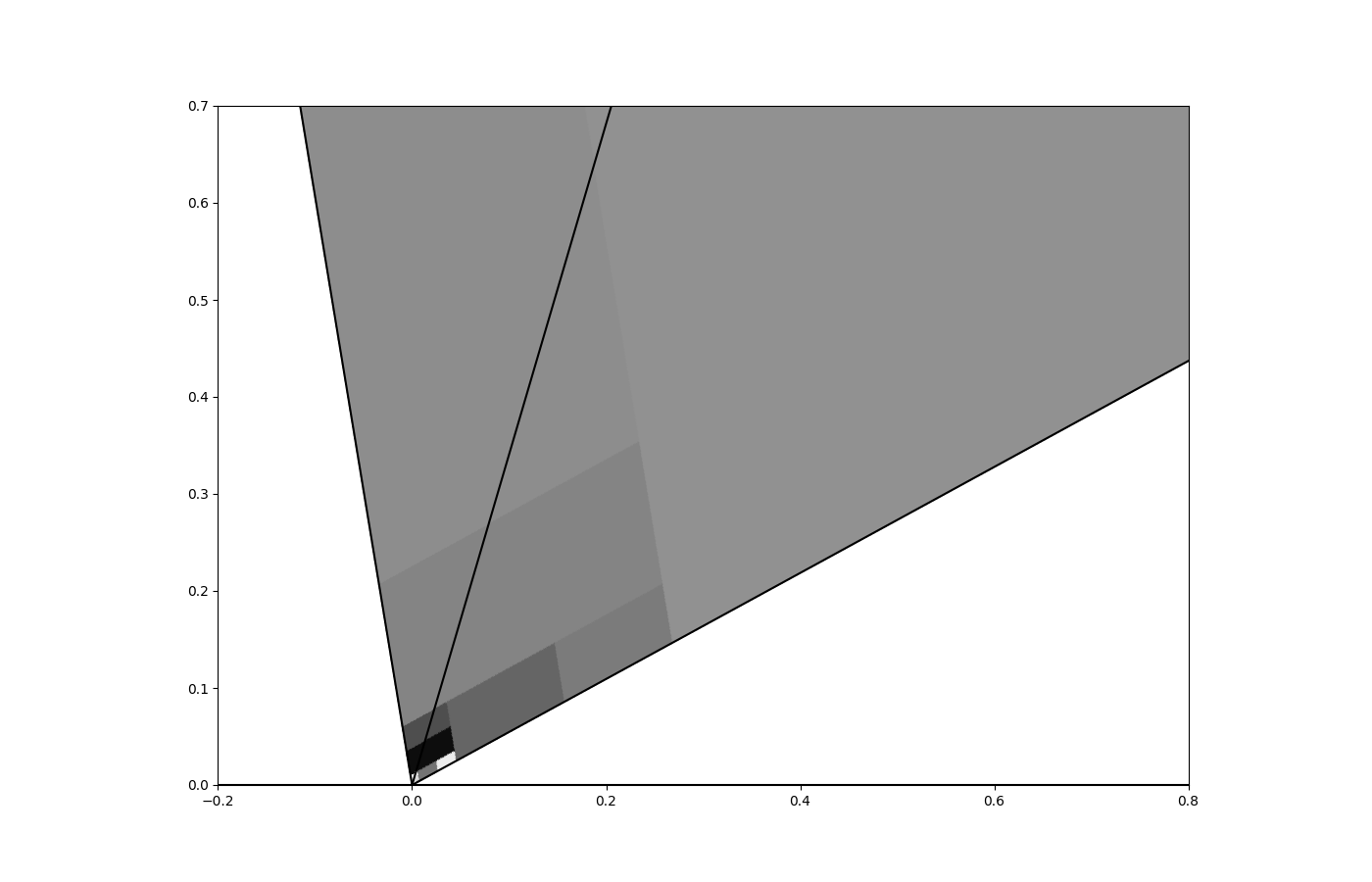}}
    \caption{An illustration of the $E$-image of the partition for the first return map $R_\kappa$ of the TCE $F_\kappa$ in figure \ref{AsymmPlot1}.  Observe the alternating stacks of parallelograms decreasing in size and cascading towards the origin.}
    \label{AsymmPart1}
\end{figure}

This theorem seems to suggest the possibility of infinite renormalizability of the first return maps to $C_c$ for a whole class of TCEs.  At the very least, if indeed the first return map of a TCE to $C_c$ is an  isometry on $S_{m,n}(\lambda)$ for $(m,n) \in \mathcal{N}_\lambda^<$ with $m \geq m_0$ and some $m_0 \in \N$, then at the very least the partition matches with its potential renormalization.

\begin{figure}
    \centering
    \includegraphics[width=\linewidth]{{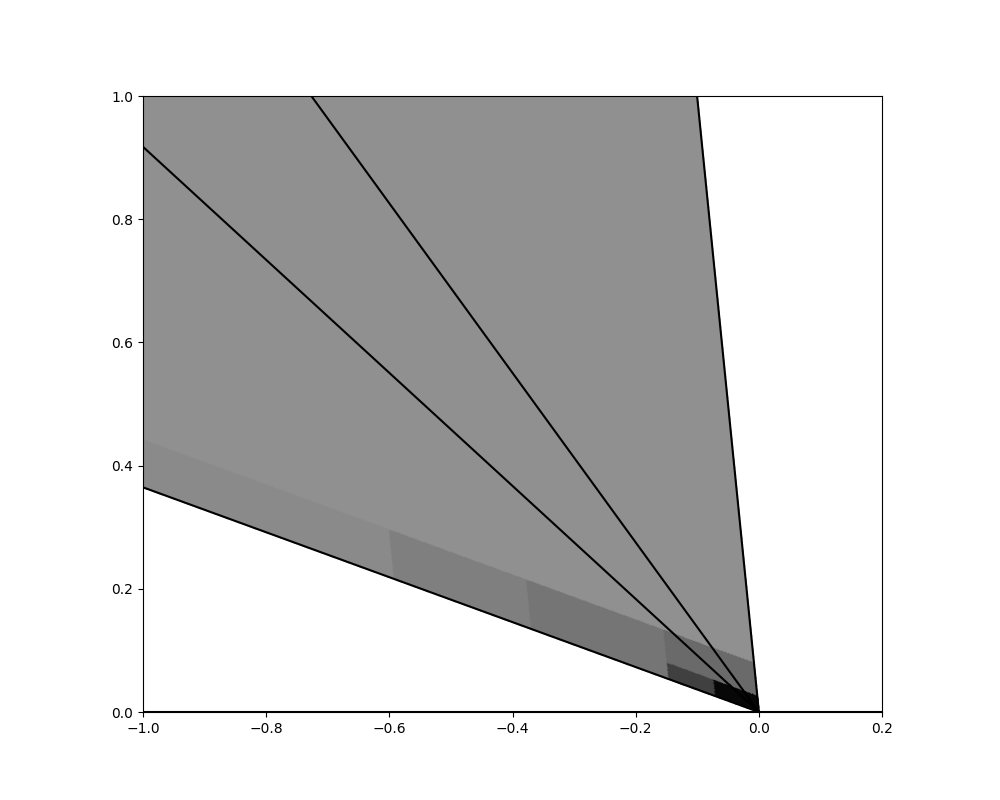}}
    \caption{A similar illustration to figure \ref{AsymmPart1} of the $E$-image of the partition for $R_\kappa$, but with the parameters from figure \ref{AsymmPlot2}.}
    \label{fig:enter-label}
\end{figure}

\section{Renormalization around zero}

In this section we investigate renormalizability of TCEs around the origin for a broad range of values of $\lambda$ and $\eta$.

Let $\lambda \in [0,1) \setminus \mathbb{Q}$ and $\eta \in \R$ such that $-\lambda < \eta = p - q\lambda < 1$ for some $p,q \in \N$.  Note that $\mathcal{N}_\lambda^<$ has a well-ordering $<'$ induced by the indexing function $w_\lambda$ so that
\begin{equation*}
(m,n) <' (r,s) \text{ if and only if } w_\lambda(m,n) < w_\lambda(r,s).
\end{equation*}
Thus, the notion of a ``maximal element'' of a finite subset of $\mathcal{N}_\lambda^<$ is well-defined.

Let $(m_0,n_0)$ be the largest element of $\mathcal{N}_\lambda^<$ such that 
\begin{equation}
\label{m0n0}
P_{m_0,n_0} < p \text{ or } Q_{m_0,n_0} < q.  
\end{equation}
The pair $(m_0,n_0)$ is well-defined since $p, q \geq 1$ but $P_{0,0} = 0$.  Thus, $w_\lambda(m_0,n_0) \geq w_\lambda(0,0)$.

Note that for all $(m,n) \in \mathcal{N}_\lambda^<$,
\begin{equation*}
 \Delta_{m,n} = Q_{m,n}\lambda - P_{m,n} = -\eta + (Q_{m,n} - q)\lambda - (P_{m,n} - p).
\end{equation*}
Define the sequence $(h_{m,n})_{(m,n) \in \mathcal{N}_\lambda^<}$ of positive integers by
\begin{equation}
\label{hmn}
h_{m,n} = (Q_{m,n} - q) + (P_{m,n} - p) + 1.
\end{equation}
We can establish some recurrence relations for the sequence $h_{m,n}$ using those of $P_{m,n}$ and $Q_{m,n}$.

\begin{proposition}
Let $(m,n) \in \mathcal{N}_\lambda^<$.  Then
\begin{equation*}
h_{m,n+1} = (n+1)h_{m,0} + h_{m-1,0} + (n+1)(p + q - 1).
\end{equation*}
Moreover, if $w_\lambda(m,n) > w_\lambda(m_0,n_0)$, then $h_{m,n} > 0$.
\end{proposition}
\begin{proof}
Recall from the definition of $P_{m,n}$ and $Q_{m,n}$ in \eqref{onesideconv} that
\begin{equation*}
Q_{m,n+1} = (n+1)Q_{m,0} + Q_{m-1,0},
\end{equation*}
and
\begin{equation*}
P_{m,n+1} = (n+1)P_{m,0} + P_{m-1,0}.
\end{equation*}
By applying this to the formula \eqref{hmn}, we get
\begin{align*}
h_{m,n+1} 	&= (Q_{m,n+1} - q) + (P_{m,n+1} - p) + 1 \\
			&= ((n+1)Q_{m,0} + Q_{m-1,0} - q) + ((n+1)P_{m,0} + P_{m-1,0} - p) + 1 \\
			&= (n+1)(Q_{m,0} + P_{m,0}) + \left( (Q_{m-1,0} - q) + (P_{m-1,0} - p) + 1 \right).
\end{align*}
Recalling the formula \eqref{hmn} for $h_{m,0}$, we get
\begin{align*}
h_{m,n+1} 	&= (n+1)((Q_{m,0} - q) + (P_{m,0} - p) + 1 + (p + q - 1)) + h_{m-1,0} \\
			&= (n+1)h_{m,0} + h_{m-1,0} + (n+1)(p + q - 1).
\end{align*}
Recall that $(m_0,n_0)$ is the largest pair in $\mathcal{N}_\lambda^<$ such that either $Q_{m_0,n_0} < q$ or $P_{m_0,n_0} < p$.  Thus, if $w_\lambda(m,n) > w_\lambda(m_0,n_0)$, then necessarily $Q_{m,n} \geq q$ and $P_{m,n} \geq p$, which further implies $h_{m,n} > 0$ by \eqref{hmn}.
\end{proof}

We won't attempt to find recursive relations for all iterates of $F_\kappa$ at 0, but we will at least calculate the orbit of 0 at iterates given by the sequence $(h_{m,n})_{(m,n) \in \mathcal{N}_\lambda^<}$.

\begin{lemma}
\label{orbitzero}
Let $(m,n) \in \mathcal{N}_\lambda^<$ such that $w_\lambda(m,n) > w_\lambda(m_0,n_0)$.  Then
\begin{equation*}
F_\kappa^{h_{m,n}}(0) = \Delta_{m,n}.
\end{equation*}
\end{lemma}
\begin{proof}
Suppose, for a contradiction, that $F_\kappa^{h_{m,n}}(0) \neq \Delta_{m,n}$.  Let $(a_t)_{t \in \N}$ and $(b_t)_{n \in \N}$ denote the sequences defined by
\begin{equation}
\label{atbt}
F_\kappa^t(0) = -\eta + b_t\lambda - a_t.
\end{equation}
Note that since $\eta = p - q\lambda$, we have
\begin{equation*}
F_\kappa^t(0) = q\lambda - p + b_t\lambda - a_t = (b_t + q)\lambda - (a_t + p),
\end{equation*}
which is never equal to 0 since $\lambda$ is irrational and $b_t+q$ and $a_t+p$ are both positive.  Observe that $(a_t)_t$ and $(b_t)_t$ are both non-decreasing and obey the following rule:
\begin{equation*}
(a_{t+1},b_{t+1}) = 	\begin{cases}
				(a_t, b_t + 1), &\text{ if } F_\kappa^t(0) < 0, \\
				(a_t + 1, b_t),	&\text{ if } F_\kappa^t(0) > 0.
				\end{cases}
\end{equation*}
Given that $a_0 = a_1 = b_0 = b_1 = 0$, we can deduce that the sequences $(a_t)_t$ and $(b_t)_t$ achieve every non-negative integer value, that is, for any $N \in \N$, there is some $t \in N$ such that $a_t = N$, and similarly there is some $t' \in \N$ such that $b_{t'} = N$.  Moreover, a simple inductive argument shows that for all integers $t \geq 1$,
\begin{equation*}
a_t + b_t + 1 = t.
\end{equation*}
Next, observe that $F_\kappa^{h_{m,n}}(0) \neq \Delta_{m,n}$ is equivalent to the statement that $a_{h_{m,n}} \neq P_{m,n}$ and $b_{h_{m,n}} \neq Q_{m,n}$.  However, notice that 
\begin{equation*}
(Q_{m,n} - q) + (P_{m,n} - p) + 1 = h_{m,n} = a_{h_{m,n}} + b_{h_{m,n}} + 1.
\end{equation*}
This implies one of two cases.
\begin{enumerate}[label=(\arabic*)]
\item $a_{h_{m,n}} < P_{m,n} - p$ and $b_{h_{m,n}} > Q_{m,n} - q$; or
\item $a_{h_{m,n}} > P_{m,n} - p$ and $b_{h_{m,n}} < Q_{m,n} - q$.
\end{enumerate}
Suppose case (1) holds.  Since $(b_t)_t$ is non-decreasing and takes every non-negative integer value, we know that there is some non-negative integer $t^* < h_{m,n}$ such that $b_{t^*} = Q_{m,n} - q$.  Thus,
\begin{equation*}
F_\kappa^{t^*}(0) = -\eta + (Q_{m,n} - q)\lambda - a_{t^*}.
\end{equation*}
Note that $a_{t^*} < a_t < P_{m,n} - p$.  Now, suppose $\Delta_{m,n} > 0$.  Then for any $0 \leq j \leq P_{m,n} - p - a_{t^*}$, we have
\begin{align*}
F_\kappa^{t^*}(0) - j 	&= -\eta + (Q_{m,n} - q)\lambda - (a_{t^*} + j) \\
					&> (Q_{m,n} - q)\lambda - (P_{m,n} - p) \\
					&= \Delta_{m,n} \\
					&> 0.
\end{align*}
Thus, $F_\kappa^{t^*+j}(0) = F_\kappa^{t^*}(0) - j$ for $0 \leq j \leq P_{m,n} - p - a_{t^*}$.  In particular,
\begin{equation*}
F_\kappa^{t^*+P_{m,n}-p-a_{t^*}}(0) = - \eta + (Q_{m,n} - q)\lambda - (P_{m,n} - p).
\end{equation*}
But then
\begin{align*}
t^* + P_{m,n} - p - a_{t^*} 	&= (Q_{m,n} - q) + (P_{m,n} - p) + 1 \\
						&= h_{m,n}.
\end{align*}
And this implies that
\begin{align*}
F_\kappa^{h_{m,n}}(0)  	&= F_\kappa^{t^*+P_{m,n}-p-a_{t^*}}(0) \\
				&= -\eta + (Q_{m,n} - q)\lambda - (P_{m,n} - p) \\
				&= \Delta_{m,n}.
\end{align*}
But this contradicts our assumption that $F_\kappa^{h_{m,n}}(0) \neq \Delta_{m,n}$.

Now suppose that $\Delta_{m,n} < 0$.  Let $0 \leq j < P_{m,n} - p - a_{t^*}$.  Then either $F_\kappa^{t^*}(0) - j > 0$ or $\Delta_{m,n} < F_\kappa^{t^*}(0) - j < 0$.  But $\Delta_{m,n} < F_\kappa^{t^*}(0) - j < 0$ implies
\begin{equation*}
|Q_{m,n}\lambda - (a_{t^*} + j + p)| < |\Delta_{m,n}| = |Q_{m,n}\lambda  - P_{m,n}|,
\end{equation*}
and $a_{t^*} + j + p < P_{m,n}$.  This contradicts the 'best approximate' property of the semiconvergent $P_{m,n}/Q_{m,n}$.  Thus $F_\kappa^{t^*}(0) - j > 0$ for all $0 \leq j < P_{m,n} - p - a_{t^*}$.  Thus, by a similar argument to before, we reach the contradiction that $F_\kappa^{h_{m,n}}(0) = \Delta_{m,n}$.

In case (2), $a_{h_{m,n}} > P_{m,n} - p$ and $b_{h_{m,n}} < Q_{m,n} - q$.  We can reach a similar contradiction as above, by using a similar argument where the roles of $(a_t)_t$ and $(b_t)_t$ are interchanged.

This exhausts all cases, so our assumption that $F_\kappa^{h_{m,n}}(0) \neq \Delta_{m,n}$ must be false.
\end{proof}

\begin{lemma}
\label{bound1}
Let $m \in \N$ such that $w_\lambda(m,0) > w_\lambda(m_0,n_0)$.  Then, for all $1 \leq t < h_{m+1,0}$,
\begin{equation*}
|F_\kappa^t(0)| \geq |\Delta_{m,0}|.
\end{equation*}
\end{lemma}
\begin{proof}
recall that $F_\kappa^t(0) = -\eta + b_t\lambda - a_t$, for some $a_t, b_t \in \N$.  Since $(a_t)_t$ and $(b_t)_t$ are non-decreasing and $t < h_{m+1,0}$, we have that $b_t \leq q_{m+1} - q$ and $a_t \leq p_{m+1} - p$, not both equal.  Hence either $b_t + q < q_{m+1}$ or $a_t + p < p_{m+1}$ and $b_t + q \leq q_{m+1}$.  In either case, by the best approximate property of convergents
\begin{equation*}
|F_\kappa^t(0)| = |(b_t + q)\lambda - (a_t + p)| \geq |q_m\lambda - p_m| = |\Delta_{m,0}|.
\end{equation*}
\end{proof}

\begin{lemma}
\label{ineqs}
Let $(m,n) \in \mathcal{N}_\lambda^<$.  Then for all $1 \leq t < h_{m,n+1}$,
\begin{align*}
F_\kappa^t(0) &\geq \Delta_{m,0} \text{ or } F_\kappa^t(0) \leq n\Delta_{m,0} + \Delta_{m-1,0} \text{ if $m$ is even,} \\[1mm]
F_\kappa^t(0) &\leq \Delta_{m,0} \text{ or } F_\kappa^t(0) \geq n\Delta_{m,0} + \Delta_{m-1,0} \text{ if $m$ is odd.}
\end{align*}
\end{lemma}
\begin{proof}
From \eqref{errorabove}, recall that if $m$ is even and $n > 0$, then $P_{m,n}/Q_{m,n} > \lambda$ is a best approximate from above, which implies
\begin{equation}
\label{bounderror}
b\lambda - a \leq Q_{m,n}\lambda - P_{m,n} < 0,
\end{equation}
for all $a,b \in \Z$ with $0 < b < Q_{m,n+1}$ such that $P_{m,n}/Q_{m,n} \neq a/b > \lambda$.

Let $(a_t)_{t \in \N}$ and $(b_t)_{t \in \N}$ be the sequences described by \eqref{atbt}.  Suppose that $1 \leq t \leq h_{m,n}$ with $F_\kappa^t(0) < 0$.  Then 
\begin{equation*}
b_t \leq Q_{m,n+1} - q \text{ and } a_t \leq P_{m,n+1} - p,
\end{equation*}
since $(a_t)_t$ and $(b_t)_t$ are non-decreasing.  Thus, 
\begin{equation}
\label{bound2}
b_t + q \leq Q_{m,n+1} \text{ and } a_t + p \leq P_{m,n+1},
\end{equation}
not both equal.  Therefore, from \eqref{atbt} we know that
\begin{equation*}
F_\kappa^t(0) = (b_t + q)\lambda - (a_t + p),
\end{equation*}
and by \eqref{bounderror} with \eqref{bound2}, we have
\begin{equation*}
F_\kappa^t(0) \leq 	\begin{cases}
			Q_{m,n}\lambda - P_{m,n}, 	&\text{ if } n \neq 0, \\
			q_{m-1}\lambda - p_{m-1}, 		&\text{ if } n = 0.
			\end{cases}
\end{equation*}
Recalling the definition of $\Delta_{m,n}$ as in \eqref{semiconverrors}, we get
\begin{equation*}
F_\kappa^t(0) \leq 	\begin{cases}
			\Delta_{m,n}, 	&\text{ if } n \neq 0, \\
			\Delta_{m-1,0}, 	&\text{ if } n = 0.
			\end{cases}
\end{equation*}
Finally, by using \eqref{Deltan}, we have
\begin{equation*}
F_\kappa^t(0) \leq n\Delta_{m,0} + \Delta_{m-1,0}.
\end{equation*}
If $F_\kappa^t(0) > 0$, then by Lemma \ref{bound1}, we have 
\begin{equation*}
F_\kappa^t(0) \geq \Delta_{m,0}.
\end{equation*}
From \eqref{errorbelow}, recall that if $m$ is odd and $n > 0$, then $P_{m,n}/Q_{m,n} < \lambda$ is a best approximate from below, that is
\begin{equation*}
b\lambda - a \geq Q_{m,n}\lambda - P_{m,n} > 0,
\end{equation*}
for all $a,b \in \Z$ with $0 < b < Q_{m,n+1}$ such that $P_{m,n}/Q_{m,n} \neq a/b < \lambda$.  Thus, in the case that $m$ is odd and $F_\kappa^t(0) > 0$, then 
\begin{equation*}
b_t + q \leq Q_{m,n+1} \text{ and } a_t + p \leq P_{m,n+1},
\end{equation*}
not both equal.  Thus, similarly to the above case where $m$ is even, we have
\begin{align*}
F_\kappa^t(0) 	&= (b_t + q)\lambda - (a_t + p) \\
		&\geq 	\begin{cases}
				Q_{m,n}\lambda - P_{m,n} 	&\text{ if } n \neq 0 \\
				q_{m-1}\lambda - p_{m-1} 		&\text{ if } n = 0
				\end{cases} \\
		&= n\Delta_{m,0} + \Delta_{m-1,0}.
\end{align*}
On the other hand, if $F_\kappa^t(0) < 0$, then by Lemma \ref{bound1}, 
\begin{equation*}
F_\kappa^t(0) \leq \Delta_{m,0}.
\end{equation*}
\end{proof}

In order to prove the next theorem, we will distinguish between the following two cases and we will prove them separately. We will first prove that
\begin{equation*}
\text{if } z \in E^{-1}(S_{m,n}), \text{ then } h(z) = h_{m,n+1},
\end{equation*}
and then we will prove that
\begin{equation*}
\text{if } h(z) = h_{m,n+1}, \text{ then } F_\kappa^{h(z)}(z) = E(z) + \Delta_{m,n+1}.
\end{equation*}

\begin{theorem}
\label{conje}
Let $\alpha \in \mathbb{B}^{d+2}$, $\tau: \{1,...,d\} \rightarrow \{1,...,d\}$ be a bijection, $\lambda \in [0,1) \setminus \mathbb{Q}$, $-\lambda < \eta = p - q\lambda < 1$ for some $p,q \in \N \setminus \{0\}$ and set $\kappa = (\alpha, \tau, \lambda, \eta, 1)$.  For all $(m,n) \in \mathcal{N}_\lambda^<$ with $w_\lambda(m,n) > w_\lambda(m_0,n_0)$, $h(E^{-1}(S_{m,n}))$ exists and is equal to $h_{m,n+1}$.  Moreover, let $z \in E^{-1}(S_{m,n})$.  Then
\begin{equation}
\label{NewRz}
R_\kappa(z) = E(z) + \Delta_{m,n+1}(\lambda).
\end{equation}
\end{theorem}
\begin{proof}
Let $z \in C_c$.  Assume $z \in E^{-1}(S_{m,n})$.  Observe that $E(z) + F_\kappa^{h_{m,n+1}}(0) \in C_c$, so $h(z) \leq h_{m,n+1}$.  By Lemma \ref{Split}, we know that 
\begin{equation*}
F_\kappa^{h(z)}(z) = E(z) + F_\kappa^{h(z)}(0).
\end{equation*}
Suppose, for a contradiction, that $h(z) < h_{m,n+1}$.  We will prove the contradiction for even and odd $m$ separately, starting with the case that $m$ is even.  By Lemma \ref{ineqs}, we know that either $F_\kappa^{h(z)}(0) \geq \Delta_{m,0}$ or $F_\kappa^{h(z)}(0) \leq n\Delta_{m,0} + \Delta_{m-1,0}$.  Since $m$ is even, recall that
\begin{equation*}
S_{m,n} = (C_0 - \Delta_{m,0}) \cap (C_c - \Delta_{m,n+1}) \cap C_c \cap (C_{d+1} - (n\Delta_{m,0} + \Delta_{m-1,0})).
\end{equation*}

Observe that
\begin{equation*}
\begin{aligned}
F_\kappa^{h(z)}(z) 	&\in S_{m,n} + F_\kappa^{h(z)}(0) \\
			&\subset (C_0 + (F_\kappa^{h(z)}(0) - \Delta_{m,0})) \cap (C_{d+1} + (F_\kappa^{h(z)}(0) - (n\Delta_{m,0} + \Delta_{m-1,0}))).
\end{aligned}
\end{equation*}
Therefore, if $F_\kappa^{h(z)}(0) \geq \Delta_{m,0}$, then 
\begin{equation*}
F_\kappa^{h(z)}(z) \in C_0 + (F_\kappa^{h(z)}(0) - \Delta_{m,0}) \subset C_0.
\end{equation*}
But by the definition of $h(z)$ as in \eqref{hz}, we have
\begin{equation*}
F_\kappa^{h(z)}(z) = R_\kappa(z) \in C_c,
\end{equation*}
which reveals a contradiction.  Similarly, if $F_\kappa^{h(z)}(0) \leq n\Delta_{m,0} + \Delta_{m-1,0}$, then 
\begin{equation*}
F_\kappa^{h(z)}(z) \in C_{d+1} + (F_\kappa^{h(z)}(0) -(n\Delta_{m,0} + \Delta_{m-1,0})) \subset C_{d+1},
\end{equation*}
which also contradicts $F_\kappa^{h(z)}(z) \in C_c$. \vspace{5mm}

Now suppose $m$ is odd.  Then
\begin{equation*}
S_{m,n} = (C_0 - (n\Delta_{m,0} + \Delta_{m-1,0})) \cap C_c \cap (C_c - \Delta_{m,n+1})) \cap (C_{d+1} - \Delta_{m,0}).
\end{equation*}
By Lemma \ref{ineqs}, we know that either 
\begin{equation*}
F_\kappa^{h(z)}(0) \geq n\Delta_{m,0} + \Delta_{m-1,0} \text{ or } F_\kappa^{h(z)}(0) \leq \Delta_{m,n}.
\end{equation*}
Clearly, either of these cases give similar contradictions as before.  Therefore our assumption that $h(z) < h_{m,n+1}$ must be false, so in fact $h(z) = h_{m,n+1}$. \vspace{5mm}

Now, suppose $h(z) = h_{m,n+1}$.  By Lemma \ref{Split},
\begin{equation*}
F_\kappa^{h(z)}(z) = E(z) + F_\kappa^{h(z)}(0),
\end{equation*}
and by Lemma \ref{orbitzero} we know that 
\begin{equation*}
F_\kappa^{h_{m,n+1}}(0) = \Delta_{m,n+1}.
\end{equation*}
Combining these two with our assumption gives us that if 
$h(z) = h_{m,n+1},$
then
$F_\kappa^{h(z)}(z) = E(z) + \Delta_{m,n+1}.$
\end{proof}

With this Theorem, as well as Theorem \ref{Scalingprop}, we can prove the existence of a renormalization scheme around the point 0.  First, we will find the definition of the first return map $R_\kappa$ on the rest of $C_c$.

\begin{lemma}
\label{UnionFillsPc}
Let $S_{m,n}$ be as in \eqref{defSmn} and $(m_0,n_0) \in \mathcal{N}_\lambda^<$ as in \eqref{m0n0}.  Define the set $U(\kappa)$
\begin{equation*}
U(\kappa) = \{0\} \cup \bigcup_{\substack{(m,n) \in \mathcal{N}_\lambda^< \\ w_\lambda(m,n) \geq w_\lambda(m_0,n_0)}} S_{m,n}.
\end{equation*}
Then
\begin{equation*}
U(\kappa) = 	\begin{cases}
		(C_0 - \Delta_{m_0,0}) \cap C_c \cap (C_{d+1} - (n_0\Delta_{m_0,0} + \Delta_{m_0-1,0})), 	&\text{if } m_0 \text{ is even,} \\[1.5mm]
		(C_0 - (n_0 \Delta_{m_0,0} + \Delta_{m_0-1,0})) \cap C_c \cap (C_{d+1} - \Delta_{m_0,0}), 	&\text{if } m_0 \text{ is odd,}
	\end{cases}
\end{equation*}
and $U(\kappa)$ is convex.
\end{lemma}
\begin{proof}
For brevity we will drop the parameters $\kappa$ when they are unambiguous.  We will first prove the equality
\begin{equation*}
U(\kappa) = 	\begin{cases}
	(C_0 - \Delta_{m_0,0}) \cap C_c \cap (C_{d+1} - (n_0\Delta_{m_0,0} + \Delta_{m_0-1,0})), 	&\text{if } m_0 \text{ is even,} \\[1.5mm]
	(C_0 - (n_0\Delta_{m_0,0} + \Delta_{m_0-1,0})) \cap C_c \cap (C_{d+1} - \Delta_{m_0,0}), 	&\text{if } m_0 \text{ is odd.}
	\end{cases}.
\end{equation*}
Observe that for all $(m,n) \in \mathcal{N}_\lambda^<$, 
\begin{equation*}
S_{m,n} \subset C_c.
\end{equation*}
Additionally, if $(m,n) \in \mathcal{N}_\lambda^<$ such that $w_\lambda(m,n) = j + k_0 - 1$, then
\begin{equation*}
S_{m,n} \subset 	\begin{cases}
			\bigcup_{m=m_0}^\infty (C_0 - \Delta_{m,0}), 	&\text{ if $m_0$ is even,} \\[2mm]
			\bigcup_{w_\lambda(m,n) \geq w_\lambda(m_0,n_0)} (C_0 - (n\Delta_{m,0} + \Delta_{m-1,0})), 	&\text{ if $m_0$ is odd.}
			\end{cases}
\end{equation*}
Note that $C_0 - (n\Delta_{m+1,0} + \Delta_{m,0}) \subset C_0 - \Delta_{m,0}$ for all $(m+1,n) \in \mathcal{N}_\lambda^<$.  Thus, we have
\begin{equation*}
S_{m,n} \subset 	\begin{cases}
			C_0 - \Delta_{m_0,0},						&\text{ if $m_0$ is even,} \\[1.5mm]
			C_0 - (n_0\Delta_{m_0,0} + \Delta_{m_0-1,0}), 	&\text{ if $m_0$ is odd.}
			\end{cases}.
\end{equation*}
We also know that
\begin{equation*}
S_{m,n} \subset 	\begin{cases}
			\bigcup_{w_\lambda(m,n) \geq w_\lambda(m_0,n_0)} (C_{d+1} - (n\Delta_{m,0} + \Delta_{m-1,0})), 	&\text{ if $m_0$ is even,} \\[2mm]
			\bigcup_{m=m_0+1}^\infty (C_{d+1} - \Delta_{m,0}), 	&\text{ if $m_0$ is odd.}
			\end{cases}
\end{equation*}
Thus, with a similar argument as above, we can show that
\begin{equation*}
S_{m,n} \subset 	\begin{cases}
			C_{d+1} - (n_0\Delta_{m_0,0} + \Delta_{m_0-1,0}),	&\text{ if $m_0$ is even,} \\[1.5mm]
			C_{d+1} - \Delta_{m_0,0},						&\text{ if $m_0$ is odd.}
			\end{cases},
\end{equation*}
Altogether, we deduce that
\begin{equation*}
U(\kappa) \subset 	\begin{cases}
			(C_0 - \Delta_{m_0,0}) \cap C_c \cap (C_{d+1} - (n_0\Delta_{m_0,0} + \Delta_{m_0-1,0})), 	&\text{if } m_0 \text{ is even,} \\[1.5mm]
			(C_0 - (n_0\Delta_{m_0,0} + \Delta_{m_0-1,0})) \cap C_c \cap (C_{d+1} - \Delta_{m_0,0}), 	&\text{if } m_0 \text{ is odd.}
			\end{cases}.
\end{equation*}

Suppose $m_0$ is even, and let
\begin{equation*}
z \in (C_0 - \Delta_{m_0,0}) \cap C_c \cap (C_{d+1} - (n_0\Delta_{m_0,0} + \Delta_{m_0-1,0})),
\end{equation*}
with $z \neq 0$.  Then there is some $(m,n) \in \mathcal{N}_\lambda^<$ with $m$ is odd and $w_\lambda(m,n) \geq w_\lambda(m_0,n_0)$ such that 
\begin{equation*}
z \in C_0 - (n\Delta_{m,0} + \Delta_{m-1,0}),
\end{equation*}
and there is some $(m',n') \in \mathcal{N}_\lambda^<$ with $m'$ even and $w_\lambda(m',n') \geq w_\lambda(m_0,n_0)$ such that 
\begin{equation*}
z \in C_{d+1} - (n\Delta_{m,0} + \Delta_{m-1,0}).
\end{equation*}
Suppose that $(m,n)$ and $(m',n')$ are the largest such pairs, which is well-defined since $z \neq 0$.  Since $m$ is odd and $m'$ is even, we either have $m < m'$ or $m' < m$.  Suppose $m < m'$.  Then
\begin{equation*}
z \in (C_0 - (n\Delta_{m,0} + \Delta_{m-1,0})) \cap C_c \cap (C_{d+1} - \Delta_{m,0}).
\end{equation*}
Since $(m,n)$ is the largest pair such that $z \in C_0 - (n\Delta_{m,0} + \Delta_{m-1,0})$, and noting that $(n+1)\Delta_{m,0} + \Delta_{m-1,0} = \Delta_{m,n+1}$ since $n+1>0$, we have that
\begin{equation*}
z \in (C_c - \Delta_{m,n+1}) \text{ or } z \in (C_{d+1} - \Delta_{m,n+1}).
\end{equation*}
However, $C_{d+1} - \Delta_{m,n+1} \subset C_{d+1}$ since $m$ is odd and so $\Delta_{m,n+1} > 0$.  Importantly, 
\begin{equation*}
C_c \cap (C_{d+1} - \Delta_{m,n+1}) = \emptyset,
\end{equation*}
and thus $z \in C_c - \Delta_{m,n+1}$.  Therefore,
\begin{equation*}
z \in (C_0 - (n\Delta_{m,0} + \Delta_{m-1,0})) \cap (C_c - \Delta_{m,n+1}) \cap C_c \cap (C_{d+1} - \Delta_{m,0}) = S_{m,n},
\end{equation*}
so clearly $z \in U(\kappa)$.  In the case that $m' < m$ we can use a similar argument to prove that
\begin{equation*}
\begin{aligned}
z 	&\in (C_0 - \Delta_{m',0}) \cap C_c \cap (C_c - \Delta_{m',n'+1}) \cap (C_{d+1} - (n'\Delta_{m',0} + \Delta_{m'-1,0})) \\
	&= S_{m',n'},
\end{aligned}
\end{equation*}
so that $z \in U(\kappa)$.  If $m_0$ is odd and
\begin{equation*}
z \in (C_0 - (n_0\Delta_{m_0,0} + \Delta_{m_0-1,0})) \cap C_c \cap (C_{d+1} - \Delta_{m_0,0}),
\end{equation*}
with $z \neq 0$, then we can use a similar argument to prove that $z \in U(\kappa)$.  Hence, we have
\begin{equation*}
U(\kappa) = 	\begin{cases}
	(C_0 - \Delta_{m_0,0}) \cap C_c \cap (C_{d+1} - (n_0\Delta_{m_0,0} + \Delta_{m_0-1,0})) 	&\text{if } m_0 \text{ is even} \\[1.5mm]
	(C_0 - (n_0\Delta_{m_0,0} + \Delta_{m_0-1,0})) \cap C_c \cap (C_{d+1} - \Delta_{m_0,0}) 	&\text{if } m_0 \text{ is odd}
	\end{cases}.
\end{equation*}
To show that $U$ is convex, one must note that the cones $C_0$, $C_c$, and $C_{d+1}$ and all their translates are convex sets, and that the intersection of convex sets is also convex.
\end{proof}

Define 
\begin{equation}
U_{k,l}(\kappa) = \{0\} \cup \bigcup_{\substack{(m,n) \in \mathcal{N}_\lambda^< \\ w_\lambda(m,n) \geq w_\lambda(k,l)}} S_{m,n},
\end{equation}
for $w_\lambda(k,l) \geq w_\lambda(m_0,n_0)$ (omitting the arguments where unambiguous), and let \begin{equation*}
\kappa' = (\alpha, \tau, \lambda, \eta', \rho),
\end{equation*}
where $-\lambda < \eta' = p' - q'\lambda < \rho$ is such that $p' \leq p$ and $q' \leq q$.  If $(m_0', n_0') \in \mathcal{N}_\lambda^<$ is the maximal element of $\mathcal{N}_\lambda^<$ such that $P_{m_0',n_0'} < p'$ or $Q_{m'_0,n'_0} < q'$, then $w_\lambda(m_0',n_0') \leq w_\lambda(m_0,n_0)$.  Furthermore,
\begin{equation}
\label{intersectionsequal}
R_\kappa |_{U_{m_0,n_0}(\kappa)}(z) = R_{\kappa'} |_{U_{m_0,n_0}(\kappa)}(z).
\end{equation}

Given $\alpha \in \mathbb{B}$, $\tau: \{1,...,d\} \rightarrow \{1,...,d\}$ be a bijection, $\lambda \in [0,1) \setminus \mathbb{Q}$, $\rho = 1$, $-\lambda < \eta = p - q\lambda < 1$ for some $p,q \in \N$, let $(m_0,n_0) \in \mathcal{N}_\lambda^<$ be as in \eqref{m0n0}.  Let $p', q' \in \N \setminus\{0\}$ be defined by
\begin{equation*}
p' = P_{m_0,0}(g^2(\lambda)), \text{ and } q' = Q_{m_0,0}(g^2(\lambda)),
\end{equation*}
and let
\begin{equation}
\label{neweta}
\eta' = p' - q'g^2(\lambda).
\end{equation}
Clearly by the definition of the one-sided convergents in \eqref{onesideconv}, we have 
\begin{equation*}
-g^2(\lambda) < \eta < 1.    
\end{equation*}

With this in mind, we have the following Theorem.

\begin{theorem}
\label{Renorm}
Let $\alpha \in \mathbb{B}^{d+2}$, $\tau: \{1,...,d\} \rightarrow \{1,...,d\}$ be a bijection, $\lambda \in [0,1) \setminus \Q$ and $-\lambda < \eta = p - q\lambda < 1$ for some $p,q \in \N$. Set the tuples $\kappa = (\alpha,\tau,\lambda,\eta,1)$ and $\kappa' = (\alpha,\tau,\lambda',\eta',1)$, where $\lambda' = g^2(\lambda)$ and $\eta'$ is as in \eqref{neweta}.  Then for all $z \in U_{m_0,0}(\kappa')$,
\begin{equation*}
R_{\kappa'}(z) = \frac{1}{1 - \lambda_1\lambda} R_\kappa((1 - \lambda_1\lambda)z).
\end{equation*}
\end{theorem}
\begin{proof}
We begin by noting that the quantity
\begin{equation}
\label{mnmin2}
(m_0',n_0') = \max\{(m,n) \in \mathcal{N}_{g^2(\lambda)}^< : P_{m,n}(g^2(\lambda)) < p' \text{ or } Q_{m,n}(g^2(\lambda)) < q' \},
\end{equation}
satisfies $(m_0',n_0') = (m_0,0)$ and, importantly for the reasons of \eqref{intersectionsequal}, we have
\begin{equation}
\label{winequality}
w_\lambda(m_0'+2,n_0') \leq w_\lambda(m_0+2,0).
\end{equation}
Here we recall the equality in Proposition \ref{interchange} since $(m_0',n_0')$ is being considered an element of $\mathcal{N}_{g^2(\lambda)}^<$ and $(m_0,0)$ an element of $\mathcal{N}_\lambda^<$.

Recall that by Theorem \ref{Scalingprop}, we have
\begin{equation*}
\frac{1}{1 - \lambda_1\lambda}S_{j+2,n}(\lambda) = S_{j,n}(g^2(\lambda)),
\end{equation*}
for all $(j+2,n) \in \mathcal{N}_\lambda^<$ with $w_\lambda(j+2,n) \geq w_\lambda(m_0+2,0)$, i.e. $j \geq m_0$.  Note that by Proposition \ref{interchange}, 
\begin{equation*}
w_\lambda(m_0'+2,n_0') \leq w_\lambda(m_0+2,0),
\end{equation*}
is equivalent to the statement that 
\begin{equation*}
w_{g^2(\lambda)}(m_0', n_0') \leq w_{g^2(\lambda)}(m_0, 0).
\end{equation*}
Hence,
\begin{equation*}
\frac{1}{1 - \lambda_1\lambda}U_{m_0+2,0}(\kappa) = U_{m_0,0}(\kappa').
\end{equation*}
Let $z \in C_c$ so that $(1 - \lambda_1\lambda)z \in S_{m+2,n}(\lambda)$ with $w_\lambda(m+2,n) \geq w_\lambda(m_0+2,0)$.  Also note that since $aC_j = C_j$ for all $a > 0$ and $E$ consists of rotations about 0, we have 
\begin{equation}
\label{Escaleinvariant}
E(az) = aE(z),
\end{equation} 
for $a > 0$.  Therefore $z \in S_{m,n}(g^2(\lambda))$ and by expanding $R_\kappa$ as in \eqref{NewRz} we get
\begin{equation*}
\frac{1}{1 - \lambda_1\lambda} R_\kappa\left( (1 - \lambda_1\lambda)z \right) = \frac{1}{1 - \lambda_1\lambda}\left( E((1 - \lambda_1\lambda)z) + \Delta_{m+2,n+1}(\lambda) \right).
\end{equation*}
Using \ref{Escaleinvariant}, we get
\begin{align*}
\frac{1}{1 - \lambda_1\lambda} R_\kappa\left( (1 - \lambda_1\lambda)z \right) 
													&= \frac{1}{1 - \lambda_1\lambda}\left( (1 - \lambda_1\lambda)E(z) + \Delta_{m+2,n+1}(\lambda) \right) \\
													&= E(z) + \left( -\frac{\Delta_{m+2,n+1}(\lambda)}{\lambda_1\lambda - 1} \right).
\end{align*}
Now, recall that $p_1 = 1$ and $q_1 = \lambda_1$, so that $\Delta_{1,0}(\lambda) = \lambda_1\lambda - 1$.  Using \eqref{Delta} and comparing with the formula for $R_{\kappa'}$ as in \eqref{NewRz}, we see
\begin{align*}
\frac{1}{1 - \lambda_1\lambda} R_\kappa\left( (1 - \lambda_1\lambda)z \right) 	&= E(z) + \left( -\frac{\Delta_{m+2,n+1}(\lambda)}{\Delta_{1,0}(\lambda)} \right) \\
															&= E(z) + \Delta_{m,n+1}(g^2(\lambda)) \\
															&= R_{\kappa'}(z).
\end{align*}
\end{proof}

The immediate consequence of Theorem \ref{Renorm} is that when $\lambda \in [0,1) \setminus \mathbb{Q}$ and $-\lambda < \eta = p-q\lambda < 1$, one can renormalize infinitely ``towards'' 0 in the sense that the domains of each renormalization are shrinking neighbourhoods of 0 in $C_c$.  Additionally, the renormalizations are determined by the orbit of $\lambda$ under the square of the Gauss map $g^2(x)$.  We shall now apply our results in this and the previous section to an example.

\section{An Example}
As an example inspired by \cite{PR18}, we will set $\alpha \in \mathbb{B}^{d+2}$, $\lambda = \Phi$, $\rho = 1$, and $\eta = \Phi^2$, where 
\begin{equation}
\Phi = \frac{\sqrt{5} - 1}{2}.
\end{equation}
In this case the the space $\mathcal{N}_\lambda^< = \N \times \{0\} \cong \N$ since $\lambda = [0;\bar{1}]$ and so $\lambda_k = 1$ for all $k \in \N$.  Thus,  $w_\Phi(m,n) = w_\Phi(m) = m$.  Also due to $\lambda_k = 1$ for each $k \in \N$, the semiconvergents $P_m/Q_m$ simply coincide with the convergents $p_m/q_m$, and the convergents are in this case defined by 
\begin{equation*}
\begin{aligned}
p_0 &= 0,					&&q_0 = 1, \\
p_1 &= 1,					&&q_1 = 1, \\
p_m &= p_{n-1} + p_{n-2}, 	&&q_m = q_{m-1} + q_{m-2}.
\end{aligned}
\end{equation*}
It is thus clear that $p_m = \operatorname{Fib}_m$ and $q_m = \operatorname{Fib}_{m+1}$, where $(\operatorname{Fib}_m)_{m \in \N}$ is the Fibonacci sequence with $\operatorname{Fib}_0 = 0$ and $\operatorname{Fib}_1 = 1$.

The first return times $(h_m)_{m \in \N}$ in the case of $\lambda = \Phi$ are given by
\begin{equation*}
h_m = (q_m - 1) + (p_m - 1) + 1 = \operatorname{Fib}_{m+1} + \operatorname{Fib}_m - 1 = \operatorname{Fib}_{m+2} - 1
\end{equation*}
for all $m \in \N$. Observe that $h_m = h_{m-1} + h_{m-2} + 1$ for all $m \geq 2$, and $h_0 = 2$, $h_1 = 4$.  Now note that $\eta = \Phi^2 = 1 - \Phi = 1 - \lambda$, and thus we have
\begin{equation*}
\begin{aligned}
m_0 = \max\{ m \in \N : p_m < 1 \text{ or } q_m < 1 \} = 0.
\end{aligned}
\end{equation*}
The errors of the convergents $\Delta_m(\Phi)$ are given by
\begin{equation}
\label{Phis}
\Delta_m(\Phi) = q_m\lambda - p_m = \operatorname{Fib}_{m+1}\Phi - \operatorname{Fib}_m.
\end{equation}

\begin{figure}
\centering
\includegraphics[width=\linewidth]{{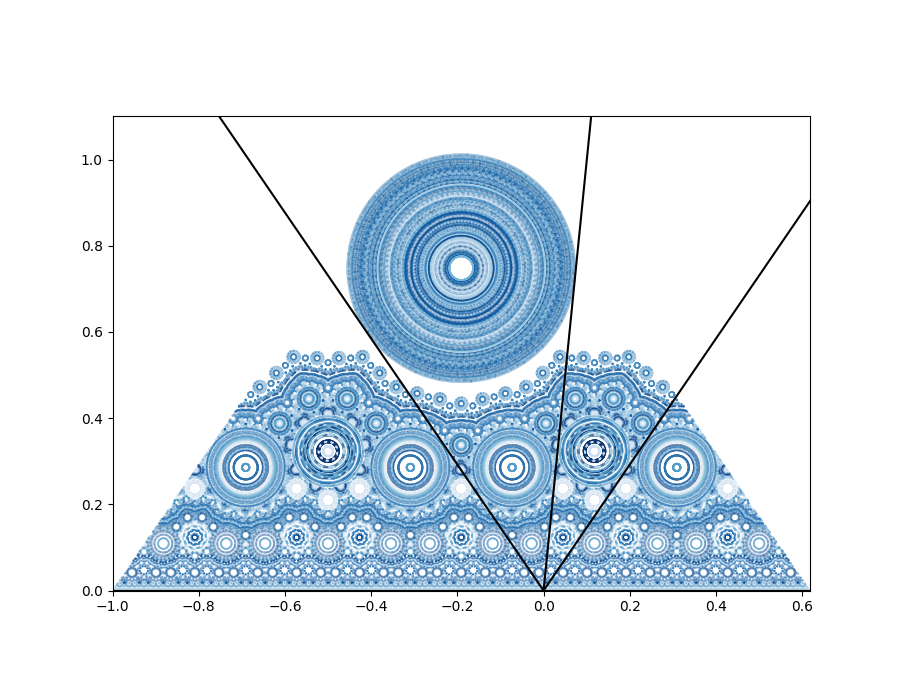}}
\caption{A plot of 1500 iterates of 500 uniformly chosen points within the box $[-1,\lambda] \times [0,1.1]$ under the TCE with parameters $\alpha = (\pi/2-0.6,0.5,0.7,\pi-0.6)$, $\tau: 1 \mapsto 2, 2 \mapsto 1$, $\lambda = \Phi$, $\eta = \Phi^2$, and $\rho = 1$.  The first 400 points of each orbit are omitted to remove transients.}
\label{plot}
\end{figure}

\begin{proposition}
We have
\begin{equation*}
\Delta_m(\Phi) = -(-\Phi)^{m+1}
\end{equation*}
\end{proposition}
\begin{proof}
Observe that
\begin{equation*}
\begin{aligned}
\Delta_m(\Phi) 	&= \operatorname{Fib}_{m+1}\Phi - \operatorname{Fib}_m \\
			&= \operatorname{Fib}_m\Phi + \operatorname{Fib}_{m-1}\Phi - \operatorname{Fib}_m \\
			&= \operatorname{Fib}_{m-1}\Phi - (1 - \Phi)\operatorname{Fib}_m \\
			&= -\Phi(\operatorname{Fib}_m\Phi - \operatorname{Fib}_{m-1}) \\
			&= -\Phi\Delta_{m-1}(\Phi).
\end{aligned}
\end{equation*}
Recall that $p_0 = 0$ and $q_0 = 1$.  Then a simple inductive argument shows us that
\begin{equation*}
\Delta_m(\Phi) = (-\Phi)^m\Delta_0(\Phi) = (-\Phi)^m(q_0\Phi - p_0) =  -(-\Phi)^{m+1}
\end{equation*}
\end{proof}

\begin{figure}
\centering
\includegraphics[width=\linewidth]{{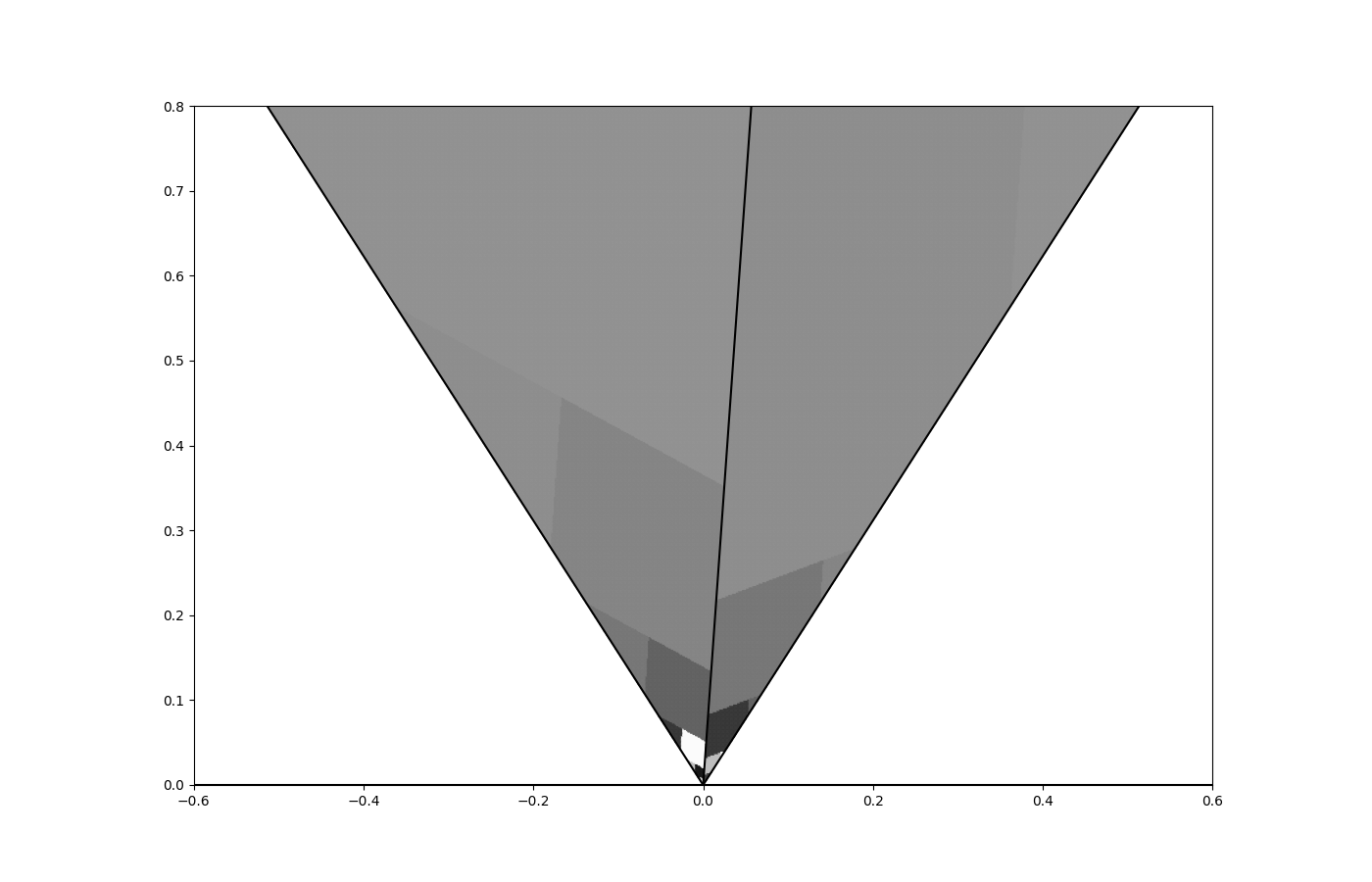}}
\caption{A partition of $C_c$ in the case where $\alpha = (1,0.5,\pi-2.5,1)$, $\tau: 1 \mapsto 2, 2 \mapsto 1$, $\lambda = \Phi$, $\eta = \Phi^2$, and $\rho = 1$.  A cascading pattern towards the origin can be seen, but its geometric structure becomes clearer after we apply the cone exchange $E$.}
\label{partpreE}
\end{figure}

Noting that the recurrence relations for $p_m$ and $q_m$ give us $p_{-1} = 1$ and $q_{-1} = 0$ and so we can set $\Delta_{-1} = -1 = -(-\Phi)^{0}$, which remains consistent with \eqref{Phis}.
With this proposition in mind, for $m \in \N$ we can determine the sets $S_m$ as
\begin{equation*}
\begin{aligned}
S_m 	&= 	\begin{cases}
			(C_0 - \Delta_m) \cap C_c \cap (C_c - \Delta_{m+1}) \cap (C_{d+1} - \Delta_{m-1}),	&\text{ if $m$ is even,} \\[1.5mm]
			(C_0 - \Delta_{m-1}) \cap (C_c - \Delta_{m+1}) \cap C_c \cap (C_{d+1} - \Delta_m),	&\text{ if $m$ is odd.}
		\end{cases} \\
	&=	\begin{cases}
			(C_0 + (-\Phi)^{m+1}) \cap C_c \cap (C_c + (-\Phi)^{m+2}) \cap (C_{d+1} + (-\Phi)^m),	&\text{ if $m$ is even,} \\[1.5mm]
			(C_0 + (-\Phi)^m) \cap (C_c + (-\Phi)^{m+2}) \cap C_c \cap (C_{d+1} + (-\Phi)^{m+1}),	&\text{ if $m$ is odd.}
		\end{cases}
\end{aligned}
\end{equation*}
These are rhombi, as can be seen in figure \ref{partpostE}, and as can be deduced from the discussion around \eqref{sidelengths} since $\alpha_0 = \alpha_{d+1}$.  It is also clear to see that for all $m \in \N$.
\begin{equation}
\label{selfsimilar}
S_{m+2} = \Phi^2 S_m.
\end{equation}
In this case, it is simple to find a partition for the entirity of the middle cone $C_c$ for the map $R_\kappa$.  In particular, define the sets $X$ and $Y$ to be
\begin{equation}
\label{X}
X = C_c \cap (C_c - (\lambda - \eta)) \cap (C_{d+1} + \eta) = C_c \cap (C_c - \Phi^3) \cap (C_{d+1} + \Phi^2),
\end{equation}
and
\begin{equation}
\label{Y}
Y = C_c \cap (C_c + \eta) = C_c \cap (C_c + \Phi^2).
\end{equation}
As we will see soon, the collection $\{X, Y, S_n : n \geq 2\}$ forms a partition of $C_c$.  We are interested in the pre-image of these sets under $E$.  In particular, define the partition 
\begin{equation*}
\mathcal{C}' = \left\{ E^{-1}(S) \cap C_j : j \in \{1, ..., d\}, S \in \{Y, X, S_2, S_3, ...,\} \right\}.
\end{equation*}

\begin{figure}
\centering
\includegraphics[width=\linewidth]{{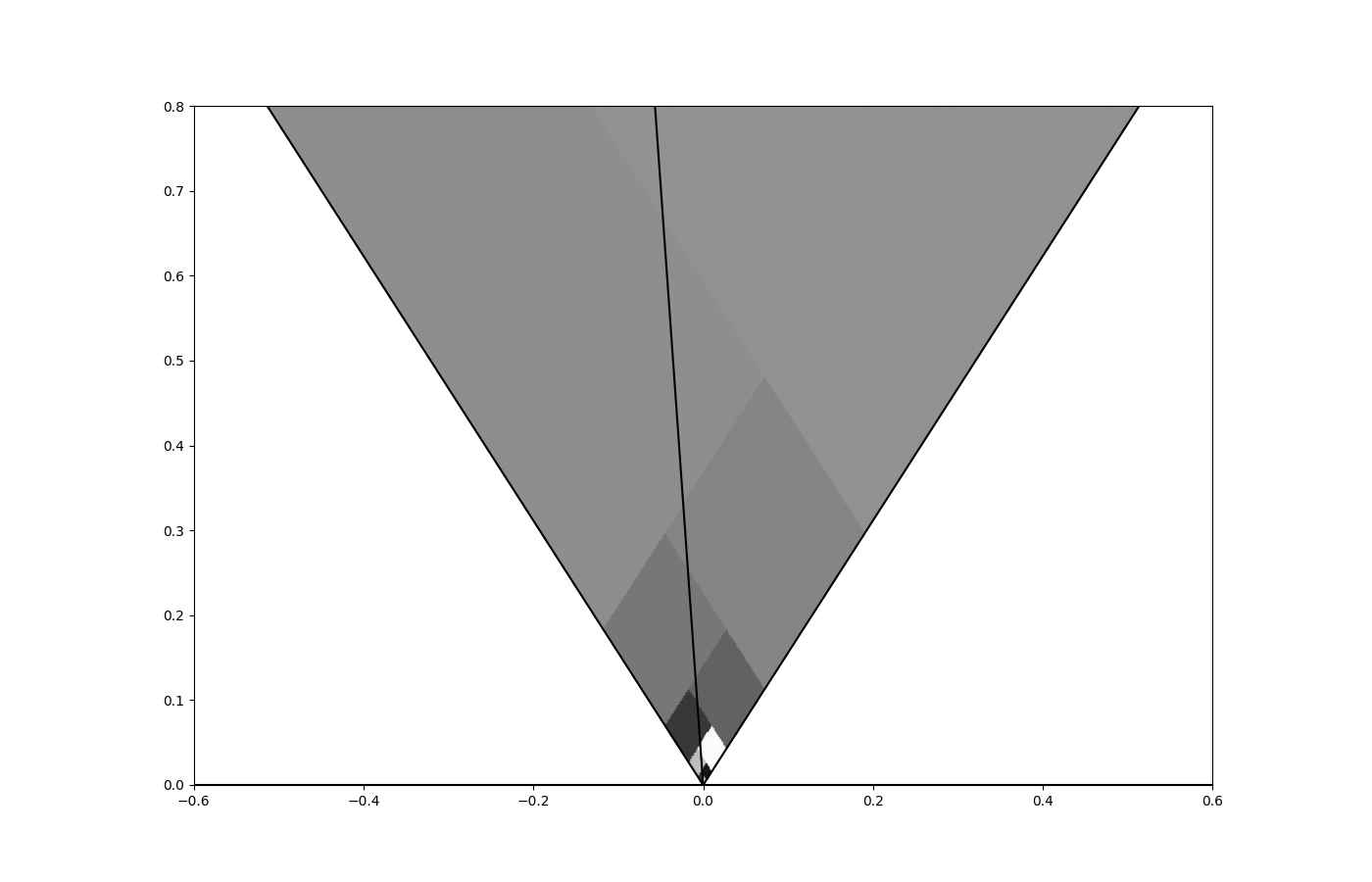}}
\caption{The same partition as in figure \ref{partpreE}, but after an application of $E$, which reveals an alternating pattern of rhombi.}
\label{partpostE}
\end{figure}

This partition can be seen in figure \ref{partpreE}.  As a consequence of the next theorem, $(\mathcal{C}', R_\kappa)$ is a PWI with a countably infinite number of atoms.  We also define a separate family of sets
\begin{equation*}
\mathcal{Q} = \left\{ Q_{n,j} = E^{-1}(S_n) \cap C_j : n \in \N, j \in \{1, ..., d\} \right\},
\end{equation*}
which includes only the rhombi, and thus forms a partition of only a subset of the middle cone.  Note that since $\lambda_m = 1$ for all $m \in \N$ and $m_0 = 0$.  We see that the set $U(\kappa)$ from Lemma \ref{UnionFillsPc} is given by
\begin{equation*}
\begin{aligned}
U(\kappa) 	&= \{0\} \cup \bigcup_{m=0}^\infty S_m \\
	&= (C_0 - \Delta_0) \cap C_c \cap (C_{d+1} - \Delta_{-1}) \\
	&= (C_0 - \Phi) \cap C_c \cap (C_{d+1} + 1).
\end{aligned}
\end{equation*}
Also observe that by removing $S_0$ and $S_1$ we get
\begin{equation*}
U_{2,0}(\kappa) = \{0\} \cup \bigcup_{m=2}^\infty S_m = (C_0 - \Phi^3) \cap C_c \cap (C_{d+1} + \Phi^2).
\end{equation*}
From this, we notice that
\begin{equation*}
U_2(\kappa) \cup X = C_c \cap (C_{d+1} + \Phi^2) \cap \left( (C_c - \Phi^3) \cup (C_0 - \Phi^3) \right),
\end{equation*}
but since $C_{d+1} - \Phi^3 \subset C_{d+1}$, we know that $C_c \cap (C_{d+1} - \Phi^3) = \emptyset$, so
\begin{align*}
U_2(\kappa) \cup X 	&= C_c \cap (C_{d+1} + \Phi^2) \cap \left( (C_{d+1} - \Phi^3) \cup (C_c - \Phi^3) \cup (C_0 - \Phi^3) \right) \\
			&= C_c \cap (C_{d+1} + \Phi^2) \cap \ol{\mathbb{H}} \\
			&= C_c \cap (C_{d+1} + \Phi^2).
\end{align*}
Therefore, we can see that
\begin{equation*}
U_2(\kappa) \cup X \cup Y = C_c \cap \left( (C_c + \Phi^2) \cup (C_{d+1} + \Phi^2) \right),
\end{equation*}
and a similar argument tells us that $C_c \cap (C_0 + \Phi^2) = \emptyset$ and so finally we get,
\begin{equation*}
U_2(\kappa) \cup X \cup Y = C_c \cap \left( (C_0  + \Phi^2) \cup (C_c + \Phi^2) \cup (C_{d+1} + \Phi^2) \right) = C_c \cap \ol{\mathbb{H}} = C_c.
\end{equation*}
Therefore, $\mathcal{C}'$ is a partition of $C_c$ up to a set of Lebesgue measure 0.

Note that for all $(m,n) \in \mathcal{N}_\Phi^<$, $w_\Phi(m,1) = w_\Phi(m+1,0)$, and so by recalling that $w_\Phi(m,0) = w_\Phi(m) = m$, the condition that $w_\Phi(m,1) > w_\Phi(m_0,0)$ is equivalent to the condition that
\begin{equation*}
w_\Phi(m+1,0) > w_\Phi(m_0,0),
\end{equation*}
which is itself equivalent to
\begin{equation*}
m > m_0 - 1 = -1.
\end{equation*}
With this in mind, Theorem \ref{conje} tells us that for all $m \in \N$, if $z \in E^{-1}(S_m)$, then
\begin{equation*}
h(z) = h_{m,1} = h_{m+1,0} = h_{m+1} = \operatorname{Fib}_{m+3} - 1,
\end{equation*}
and 
\begin{equation*}
R_\kappa(z) = E(z) + \Delta_{m+1} = E(z) - (-\Phi)^{m+2},
\end{equation*}

\begin{figure}
\centering
\includegraphics[width=\linewidth]{{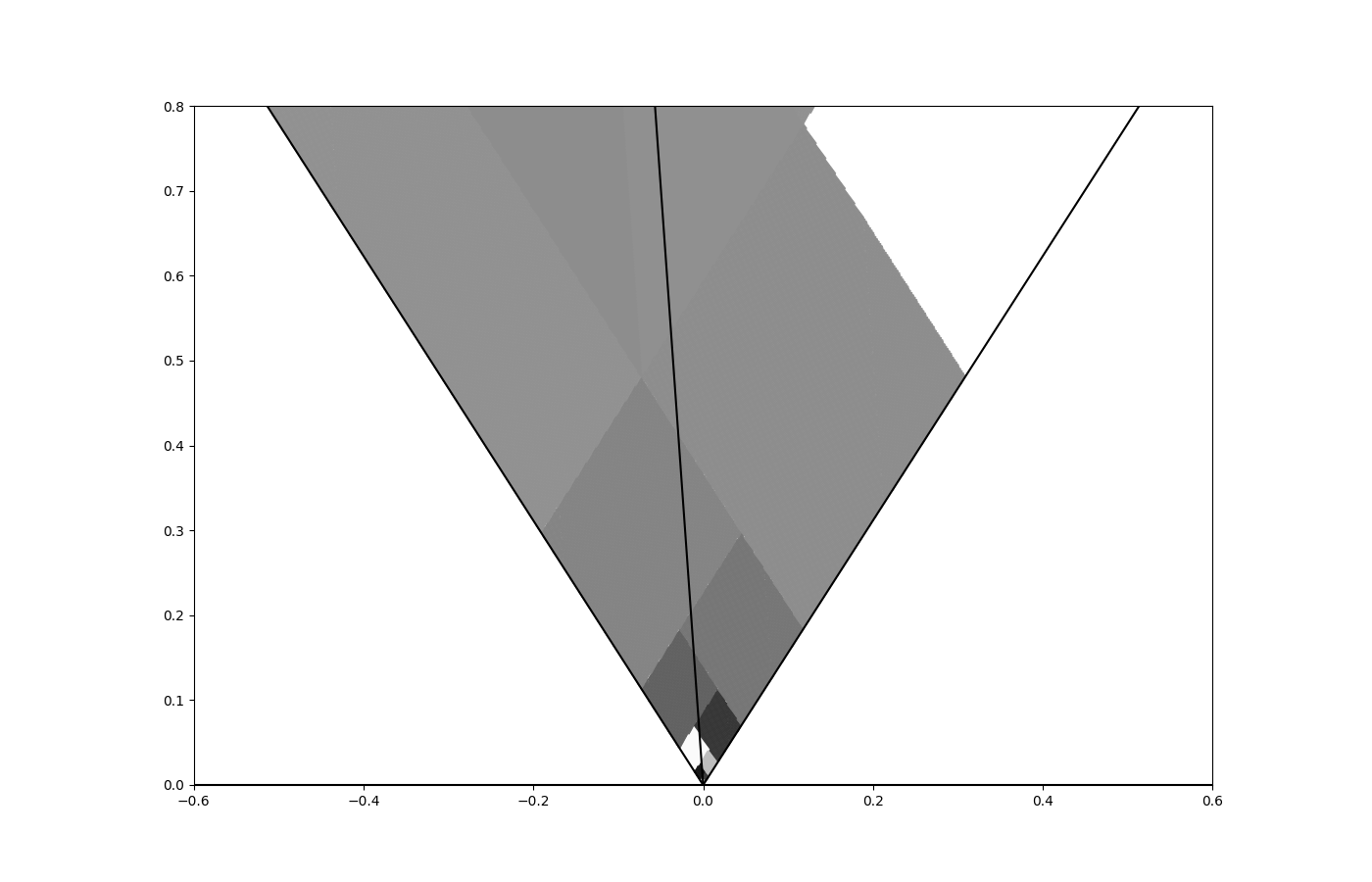}}
\caption{The same partition of $C_c$ as in figure \ref{partpreE}, after an application of $R_\kappa$, which has shifted the rhombi alternately.  Note that there is an overlap between the cone and ribbon (both of which are more clearly seen in figure \ref{partpostE}) on the top of the figure, causing an unavoidable overlap of the colours.}
\label{partpostR}
\end{figure}

Observe that $\lambda = \Phi$ is a special case of irrational number within $[0,1]$ in the sense that it is a fixed point of the Gauss map $g$.  Thus, $g^2(\Phi) = \Phi$ and thus we can choose $\eta' = \eta = 1 - \Phi = \Phi^2$ and Theorem \ref{Renorm} tells us that the first return map $R_\kappa$ exhibits exact self-similarity within $U(\kappa)$.  In particular, for all $z \in U_{0,0}(\kappa) = U(\kappa)$, we have the following conjugacy
\begin{equation*}
R_\kappa(z)  = \frac{1}{-\Delta_1}R_\kappa((-\Delta_1)z) = \frac{1}{\Phi^2}R_\kappa(\Phi^2z).
\end{equation*}

One consequence of this is that if there exists a periodic point $z \in S_m$ of period $k$, then $z$ is a periodic point of $R_\kappa$ with period $k/h_{m+1}$.  The self-similarity shows that for all $n \in \Z$ such that $2n \geq -m$, $\Phi^{2n}z$ is a periodic point of $R_\kappa$, thus also a periodic point of $F_\kappa$ whose period is an integer multiple of $h_{m+2n+1}$.  In particular, there is a sequence $(z_n)_{n \in \N}$ given by
\begin{equation}
\label{periodic_sequence}
z_n = \Phi^{2n-m}z, 
\end{equation}
so that for all $n \in \N$, $z_n \in S_{2n+\tilde{m}}$ and the period of $z_n$ is an integer multiple of $h_{2n+\tilde{m}+1}$, where $\{0,1\} \ni \tilde{m} \cong m \pmod{2}$.

Given a map $f : X \rightarrow X$, let $\mathcal{O}_f^+(x)$ denote the forward orbit of $x \in X$ under $f$, that is
\begin{equation*}
\mathcal{O}_f^+(x) = \{ f^n(x): n \in \N\}.
\end{equation*}

\begin{proposition}
\label{Application}
Suppose there exists a periodic point $z \in S_m$ for some $m \in \N$, and let $(z_n)_{n \in \N}$ be the sequence of periodic points given by \eqref{periodic_sequence}.  Then the sequence $\left( \mathcal{O}_{F_\kappa}^+(z_n) \right)_{n \in \N}$ of periodic orbits accumulates on the interval $[-1,\Phi]$.
\end{proposition}

\begin{proof}
Let $n \in \N$.  Note that 
\begin{equation}
\{F_\kappa^j(z_n) : 1 \leq j \leq h(z)\} \subset \mathcal{O}_{F_\kappa}^+(z_n).
\end{equation}
Lemma \ref{Split} tells us that for all $1 \leq j \leq h(z_n)$,
\begin{equation*}
F_\kappa^j(z_n) = E(z_n) + F_\kappa^j(0).
\end{equation*}
Therefore,
\begin{equation}
\label{dist}
|F_\kappa^j(z_n) - F_\kappa^j(0)| = |E(z_n)| = |z_n|.
\end{equation}
Let $H \in \N$.  Then there exists an $N \in \N$ such that $h(z_n) \geq H$ for all $n \geq N$, and thus \eqref{dist} holds for all $1 \leq j \leq H$.  Now let $\varepsilon > 0$ be small.  Then there exists an $N' \in \N$ such that for all integers $n \geq N'$ such that 
\begin{equation}
\label{bound}
|z_{n}| < \varepsilon.
\end{equation}
Set $N^* = \max\{N,N'\}$.  Then for all integers $n \geq N^*$, both \eqref{dist} holds for all $1 \leq j \leq H$ and \eqref{bound} holds.  Hence, for all $n \geq N^*$ we have
\begin{equation*}
|F_\kappa^j(z) - F_\kappa^j(0)| < \varepsilon,
\end{equation*}
for all $1 \leq j \leq H$.  Since $H$ and $\varepsilon$ are independent and arbitrary, we conclude that the sequence $\left( \mathcal{O}_{F_\kappa}^+(z_n) \right)_{n \in \N}$ accumulates on the set $\mathcal{O}_{F_\kappa}^+(0)$.  

By Proposition \ref{FisG}, $F_\kappa$ is a 2-IET everywhere on the interval $[-1,\lambda]$  except on the preimages of $0$, since $F_\kappa(0) = -\eta = 1 - \lambda$, contrary to $F_\kappa(x) = x + \lambda$ for $x \in [-1,0)$ and $F_\kappa(x) = x - 1$ for $x \in (0,\lambda)$.  

Therefore $F_\kappa$ is conjugate to an irrational rotation almost everywhere (with respect to one-dimensional Lebesgue measure), since $\lambda = \Phi$ is irrational.  In particular, since $\lambda$ is irrational and $\eta = 1 - \lambda$, we know, by for example Lemma \ref{bound1}, that $F_\kappa^j(0) = F_\kappa^{j-1}(-\eta)$ is bounded away from $0$ for all integers $j > 0$, so $F_\kappa^j(0) \neq 0$ for any $j > 0$.  

Hence, the orbit of $F_\kappa(0) = -\eta$ under $F_\kappa$ is also the orbit under an irrational rotation, and thus the orbit of $0$ is dense in the interval $[-1,\lambda]$, i.e.
\begin{equation*}
\ol{\mathcal{O}_{F_\kappa}^+(0)} = [-1,\lambda].
\end{equation*}
Therefore, the sequence $\left( \mathcal{O}_{F_\kappa}^+(z_n) \right)_{n \in \N}$ accumulates on the interval $[-1,\lambda]$.
\end{proof}

\begin{remark}
Although extending Proposition \ref{Application} to periodic continued fractions $\lambda$ should follow from a similar strategy to the proof used here, an extension to aperiodic continued fractions seems to require nothing short of assuming/proving that every TCE has at least one periodic point in its `renormalizable domain' $U(\kappa)$.
\end{remark}

\section{Discussion}

Translated cone exchanges, introduced first in \cite{AGPR} and investigated in \cite{P19,PR18}, are an interesting and largely unexplored family of parametrised PWIs. They contain an embedding of a simple IET on the baseline and as such they are an interesting tool to understand more general PWIs by gaining leverage from known results for IETs.
In this paper we go beyond results in \cite{AGPR,P19,PR18} to show that for a dense subset of an open set in the parameter space of TCEs there is a mapping ($\kappa \mapsto \kappa'$ in Theorem \ref{Renorm}) that determines a renormalization scheme for the first return map $R_\kappa$ of $F_\kappa$ to the vertex $0$ of the middle cone $C_c$.  This helps us describe the small-scale, long-term behaviour of $F_\kappa$ near the baseline $[-1,\lambda]$ via the large-scale, short-term behaviour of $F_{\kappa''}$ with $\kappa'' = (\alpha, \tau, \lambda'', \eta'', 1)$, $\lambda'' = g^{2k}(\lambda)$, a large enough integer $k > 0$ and some suitably chosen $\eta''$ described by \eqref{neweta}.  Proposition \ref{Application} is an example of this, where a periodic disk of small period for $F_\kappa$ and the periodicity of the continued fraction coefficients of $\lambda = \Phi$ give rise to an countable collection of periodic disks of arbitrarily high period clustering on $[-1,\lambda]$, through the renormalizability established by Theorem \ref{Renorm}.

Although these results give a glimpse into the dynamics for orbits close to the baseline, there remains a lot more to do to understand the dynamics of these TCEs near general points in exceptional set, but this seems to be a far more complex task to undertake, especially as the dynamics near the baseline primarily consists of horizontal translations, whereas in general the effect of the rotations will be inextricably linked to translations.

\subsection*{Acknowledgements}

We thank Pedro Peres and Arek Goetz for discussions about this research. NC and PA thank the Mittag-Leffler Institute for their hospitality and support to visit during the ``Two Dimensional Maps'' programme of early 2023.

\subsection*{Funding Acknowledgement}
This work was supported by the Engineering and Physical Sciences Research Council.

\subsection*{Data Access Statement}
For the purpose of open access, the authors have applied a Creative Commons Attribution (CC BY) licence to any Author Accepted Manuscript version arising from this submission.

No new data were generated or analysed during this study.  The figures in this study were produced by python programming code written by NC.  This code, namely the `pyTCE' program, is publicly available at the following link:

\centering https://github.com/NoahCockram/pyTCE/tree/main

\end{document}